%% file: affine_sos.tex
\documentclass[onefignum,onetabnum]{siamonline220329}
\input{ex_shared}

\ifpdf
\hypersetup{pdftitle={Restricted-Multiplier Positivstellensatze for Polynomial Optimization with Affine Parameters},pdfauthor={Jie Wang}}
\fi
\begin{document}
\maketitle
\begin{abstract}
We develop positivity certificates for polynomials that are affine in auxiliary parameters, with sum-of-squares multipliers depending only on the principal variables. Under a restricted Archimedean order-unit condition, every strictly positive scalar polynomial admits such a certificate. We construct a convergent restricted moment-SOS hierarchy for polynomial optimization problems with affine parameters and establish a flatness criterion for global optimality detection and global minimizer recovery. The scalar theorem extends to polynomial-matrix constraints with any number of affine parameters. For matrix-valued objectives, pointwise positivity is sufficient with one parameter but can fail with two or more parameters; uniform free positivity provides a general sufficient replacement. Applications include robust linear programming, robust Lyapunov design, min-max polynomial optimization, analysis and control of input-affine dynamical systems. Numerical experiments show that the proposed approach substantially reduces semidefinite block sizes and solution times while incurring little or no degradation in bound quality for most test cases.
\end{abstract}
\begin{keywords}
Positivstellensatz, sum of squares, affine parameter, moment-SOS hierarchy, robust optimization
\end{keywords}
\begin{MSCcodes}
Primary, 90C23; Secondary, 14P10, 13J30, 90C22
\end{MSCcodes}
\input{sections/01_introduction}
\input{sections/02_preliminaries}
\input{sections/03_scalar}

\input{sections/04_moment}
\input{sections/05_hierarchy}
\input{sections/06_extensions}
\input{sections/07_applications}
\input{sections/08_experiments}
\input{sections/10_conclusions}
% \appendix
% \input{sections/appendices}

\section*{Acknowledgements}
The author acknowledges the use of ChatGPT-6 Astra to assist with brainstorming, mathematical development, and manuscript drafting. The author is solely responsible for the final content, analysis, and conclusions.

\section*{Funding}
This work was jointly funded by the National Key R\&D Program of China under grant No.~2023YFA1009401 and the National Natural Science Foundation of China under grant Nos.~12201618 and 12171324.

% The authors declare that they have no conflicts of interest.

% \section*{Data availability}
% The authors confirm that all data generated or analyzed during this study are included in this article.

\bibliographystyle{siamplain}
\bibliography{references}
\end{document}

%% file: ex_shared.tex
\usepackage{amsfonts,amssymb,mathtools,mathrsfs,booktabs,graphicx}
\usepackage{xcolor}
\usepackage{enumitem}
\setlist[enumerate]{leftmargin=.35in}
\setlist[itemize]{leftmargin=.35in}

\newsiamremark{remark}{Remark}
\newsiamremark{example}{Example}
\newsiamremark{assumption}{Assumption}
\headers{SOS Certificates with Affine Parameters}{Jie Wang}
\title{Restricted-Multiplier Positivstellens\"atze for Polynomial Optimization with Affine Parameters}
\author{Jie Wang\thanks{State Key Laboratory of Mathematical Sciences, Academy of Mathematics and Systems Science, Chinese Academy of Sciences, Beijing, China (\email{wangjie212@amss.ac.cn}).}}
\newcommand{\R}{\mathbb R}
\newcommand{\N}{\mathbb N}
\newcommand{\A}{\mathcal A}
\newcommand{\V}{\mathcal V}
\newcommand{\M}{\mathcal M}
\newcommand{\Q}{\mathcal Q}
\newcommand{\I}{\mathcal I}
\newcommand{\Sym}{\mathbb S}
\newcommand{\sos}{\Sigma[x]}
\newcommand{\la}{\lambda}
\newcommand{\eps}{\varepsilon}
\newcommand{\ip}[2]{\langle #1,#2\rangle}
\newcommand{\norm}[1]{\lVert #1\rVert}
\DeclareMathOperator{\supp}{supp}
\DeclareMathOperator{\Aff}{Aff}
\DeclareMathOperator{\tr}{tr}
\DeclareMathOperator{\rank}{rank}

\DeclareMathOperator{\conv}{conv}

%% file: sections/01_introduction.tex
\section{Introduction}\label{sec:intro}
Positivstellens\"atze translate polynomial positivity on semialgebraic sets into algebraic certificates and form a central bridge between real algebraic
geometry and polynomial optimization \cite{krivine1964anneaux,Marshall2008,Putinar1993,schmudgen2017moment,stengle1974nullstellensatz}.
Combined with sums of squares (SOS), they lead to convergent hierarchies of semidefinite relaxations for global polynomial optimization \cite{Lasserre2001,parrilo2003semidefinite}. The related computational
cost, however, grows rapidly with the number of polynomial variables and the relaxation order. This has motivated substantial effort to develop structured positivity certificates with lower computational complexity; structures considered in the literature include correlative sparsity \cite{WakiEtAl2006}, term sparsity \cite{WangMagronLasserre2021,wang2022cs}, symmetry \cite{riener2013exploiting}, subalgebras \cite{Schmudgen2024}, low rank \cite{gaggioli2025global}, composition and tensor-train structure \cite{gaggioli2026composition}.

In this paper, we study polynomial positivity on sets of the form
\begin{equation}\label{eq:K}
 K\coloneqq \{(x,\la)\in\R^n\times\R^r:g_i(x,\la)\geq0,
          \ i=1,\ldots,m\},
\end{equation}
where both the polynomial to be certified and the constraint polynomials are
affine in an auxiliary parameter vector $\la$:
\begin{equation}\label{eq:affine}
 f(x,\la)=f_0(x)+\sum_{k=1}^r\la_k f_k(x),\qquad
 g_i(x,\la)=g_{i0}(x)+\sum_{k=1}^r\la_k g_{ik}(x).
\end{equation}
We call $x$ the principal variables and $\la$ the parameter variables.
This structure appears naturally in robust linear programming, robust control design, parameter-dependent feasibility, min-max polynomial optimization, and system analysis and control \cite{MillerSznaier2025}.

A standard approach is to regard $(x,\la)$ as a single block of variables and
apply Putinar's Positivstellensatz \cite{Putinar1993}. Under the usual Archimedean hypothesis,
strict positivity on $K$ then yields
\begin{equation}\label{eq:jointcert}
 f(x,\la)=\widehat\sigma_0(x,\la)+\sum_{i=1}^m\widehat\sigma_i(x,\la) g_i(x,\la),
 \qquad \widehat\sigma_i \text{ are SOS}.
\end{equation}
This representation underlies the standard joint-variable moment-SOS
hierarchy, but it does not preserve the affine parameter structure. The SOS
multipliers may contain arbitrary powers of $\la$, and their Gram matrices
are indexed by monomials in all $n+r$ variables. Consequently, the
positive semidefinite (PSD) blocks can become much larger as the number of parameters
increases.

The central question addressed in this work is whether the multipliers can instead be chosen to
depend only on the principal variables:
\begin{equation}\label{eq:restrictedintro}
 f(x,\la)=\sigma_0(x)+\sum_{i=1}^m\sigma_i(x)g_i(x,\la),
 \qquad \sigma_i \text{ are SOS}.
\end{equation}
% The corresponding restricted quadratic module is
% \[
%  \M_x(g)\coloneqq \Sigma[x]+\sum_{i=1}^m\Sigma[x]g_i
%  \subseteq
%  \V_r\coloneqq \R[x]\oplus\bigoplus_{k=1}^r\la_k\R[x].
% \]
% It is stable under multiplication by squares in $\R[x]$, but $\V_r$ is not
% an algebra. Thus the usual quadratic-module theorem in $\R[x,\la]$ does not
% directly imply \eqref{eq:restrictedintro}.
Our first main result is the following restricted-multiplier Positivstellensatz. 
\begin{theorem}[informal]
Let $f(x,\la),g_1(x,\la),\ldots,g_m(x,\la)$ be polynomials that are affine in parameter variables $\la$. Assume that a certain kind of Archimedeanity holds. If $f$ is strictly positive on $K$, then $f$ admits a representation of the form \eqref{eq:restrictedintro}.
\end{theorem}
The proof is based on an exact representation of positive linear
functionals. Let $\M_x(g)$ be the set of polynomials of the form \eqref{eq:restrictedintro}. Every normalized functional $\mathcal L$ that is nonnegative on
$\M_x(g)$ can be written as
\begin{equation}\label{eq:graphintro}
 \mathcal L(p)=\int p(x,\tau(x))\,d\mu(x),\qquad p\in\V_r\coloneqq \R[x]\oplus\bigoplus_{k=1}^r\la_k\R[x],
\end{equation}
where $\mu$ is a compactly supported probability measure and $\tau$ is a
bounded Borel map satisfying $(x,\tau(x))\in K$ for $\mu$-almost every $x$.
The affine dependence on $\la$ is decisive: only the conditional barycenter
of the parameter distribution is visible, and each parameter fiber is
convex. Combining this graph representation with order-unit separation gives
the positivity theorem.

The same representation solves the associated full moment problem. In addition to the ordinary moments of $x$, only the mixed moments of $x^\alpha\la_k$ are required; moments containing products such as
$\la_j\la_k$ do not appear. We characterize all representing measures by their common principal-variable marginal and conditional parameter
barycenter. Building on these results, we construct a
convergent restricted moment-SOS hierarchy for polynomial optimization problems with affine parameters, and provide a directly checkable flatness criterion for detecting global optimality and recovering global minimizers.

We then extend the scalar polynomial theory to the setting of polynomial-matrix constraints with any number of affine parameters.
\begin{theorem}[informal]
Let $f(x,\la)$ be a scalar polynomial and $G_1(x,\la),\ldots,G_m(x,\la)$ be polynomial matrices that are affine in the parameter variables $\la$. Assume that a certain kind of Archimedeanity holds. If $f$ is strictly positive on the set defined by $G_i(x,\la)\succeq0$, $i=1,\ldots,m$, then $f$ admits a matrix restricted-multiplier representation.
\end{theorem}

Surprisingly, the case of matrix-valued objectives exhibits a sharp
distinction. With one affine parameter, pointwise positivity can still guarantee a restricted matrix certificate. 
\begin{theorem}[informal]
Let $F(x,\la)$, $G_1(x,\la),\ldots,G_m(x,\la)$ be polynomial matrices that are affine in the single parameter variable $\la$. Assume that a certain kind of Archimedeanity holds. If $F$ is strictly positive on the set defined by $G_i(x,\la)\succeq0$, $i=1,\ldots,m$, then $F$ admits a matrix restricted-multiplier representation.
\end{theorem}

Nevertheless, with two or more parameters, the direct generalization is no longer true. We therefore identify uniform free positivity as a general
sufficient replacement for ordinary pointwise positivity to obtain a restricted matrix certificate.

To illustrate the usefulness of the theoretical results, we apply the restricted-multiplier Positivstellens\"atze to robust linear programming, robust Lyapunov and dissipation inequalities, min-max polynomial optimization, robust polynomial matrix inequalities, and analysis and control of input-affine dynamical systems. Extensive numerical experiments demonstrate the notable computational advantage of our approach. In particular, comparisons with the standard joint-variable approach show substantial reductions in PSD block sizes and solution times, with little or no loss in bound quality across most test cases.

% \vspace{1em}
% {\bf Related work.}
% Lasserre \cite{lasserre2011min} developed a moment-SOS
% method for min-max and robust polynomial optimization using polynomial lower approximation. 
% Klep and Nie \cite{KlepNie2020} proved a matrix Positivstellensatz with lifting polynomials. 
% Guo and Wang proposed a moment-SOS hierarchy for robust optimization with polynomial matrix inequalities in \cite{GuoWang2025} and established Positivstellens\"atze for polynomial matrices with universal quantifiers in \cite{GuoWangUQ2025}.
% Miller and Sznaier \cite{MillerSznaier2025} established that affine uncertainty in polynomial control problems can be eliminated through parameterized robust counterparts, with continuous and, under additional structural assumptions, polynomial conic multiplier selections.

The remainder of the paper is organized as follows. Section~\ref{sec:prelim}
introduces the ordered-module framework. Section~\ref{sec:scalar} proves the
scalar Positivstellensatz. The full and truncated moment theories are
developed in Sections~\ref{sec:moment} and~\ref{sec:hierarchy}, respectively.
Section~\ref{sec:extensions} treats polynomial matrices, and
Section~\ref{sec:applications} develops optimization and control
applications. Section~\ref{sec:experiments} reports numerical experiments,
and Section~\ref{sec:conclusion} concludes with open questions and further
directions.

%% file: sections/02_preliminaries.tex
\section{Notation and preliminaries}\label{sec:prelim}
We begin with the notation and ordered-module structure used throughout.
The main preliminary result characterizes the order-unit hypothesis by
finitely many algebraic bounds.

Let $\N$ denote the nonnegative integers and let $[r]\coloneqq \{1,\ldots,r\}$. We write $A\succeq0$ to indicate that the matrix $A$ is PSD. For $\alpha\in\N^n$, write $|\alpha|\coloneqq \sum_j\alpha_j$ and $x^\alpha\coloneqq \prod_jx_j^{\alpha_j}$. We use $\R[x]_{\leq d}$ for the polynomials of degree at most $d$ and $[x]_d$ for the column of all monomials of degree at most $d$. Its length is $N_n(d)\coloneqq \binom{n+d}{d}$. The cone of SOS polynomials is
\[
 \sos\coloneqq \left\{\sum_{j=1}^t p_j^2:t\in\N,\ p_j\in\R[x]\right\}.
\]
We write $\Sigma[x]_{\leq2d}\coloneqq \Sigma[x]\cap\R[x]_{\leq2d}$ for its truncation at degree $2d$. A polynomial belongs to this cone if and only if it equals $[x]_d^\intercal Q[x]_d$ for some $Q\succeq0$. We denote the space of real symmetric $q\times q$ matrices by $\Sym^q$ and its polynomial counterpart by $\Sym^q[x]$. For symmetric matrices of the same size, $\ip{A}{B}\coloneqq \tr(AB)$.

\subsection{Full and restricted quadratic modules}
Set $\A\coloneqq \R[x]$ and $\V_r\coloneqq \A\otimes_\R\Aff(\R^r)$. Every element of $\V_r$ has the unique form $p=p_0+\sum_k\la_kp_k$ with $p_k\in\A$. For $g=(g_1,\ldots,g_m)\in\V_r^m$, the full quadratic module generated by $g$ in $\R[x,\la]$ is
\begin{equation}\label{eq:full-module}
 \M(g)\coloneqq \left\{\sigma_0+\sum_{i=1}^m\sigma_i g_i:
                  \sigma_0,\ldots,\sigma_m\in\Sigma[x,\la]\right\},
\end{equation}
where $\Sigma[x,\la]$ is the cone of finite sums of squares of polynomials in all variables $(x,\la)$. Thus $\M(g)$ is the smallest subset of $\R[x,\la]$ containing $1,g_1,\ldots,g_m$ that is closed under addition and multiplication by squares of arbitrary polynomials in $\R[x,\la]$. Every element is nonnegative on $K$, although it need not be affine in $\la$.

Restricting the SOS multipliers to polynomials in $x$ alone gives
\begin{equation}\label{eq:module}
 \M_x(g)\coloneqq \left\{\sigma_0+\sum_{i=1}^m\sigma_i g_i:
                  \sigma_0,\ldots,\sigma_m\in\sos\right\}.
\end{equation}
It is a convex cone containing $1$ and satisfying $a^2\M_x(g)\subseteq\M_x(g)$ for every $a\in\A$. We call it a restricted quadratic module; in particular, $\M_x(g)\subseteq\M(g)\cap\V_r$. The term refers to closure under addition and multiplication by squares in $\A$; $\V_r$ is not an algebra when $r\geq1$.

\begin{definition}\label{def:order-unit}
Let $C$ be a convex cone in a real vector space $V$. An element $e\in C$ is an order unit for $C$ in $V$ if
\begin{equation}\label{eq:order-unit}
 \text{for every }v\in V\text{ there is }N>0
 \text{ such that }Ne+v,\ Ne-v\in C.
\end{equation}
Equivalently, writing $u\leq_C v$ when $v-u\in C$, every $v\in V$ satisfies $-Ne\leq_C v\leq_C Ne$ for some $N>0$. The cone need not be pointed, so $\leq_C$ is in general a preorder rather than a partial order.
\end{definition}

\begin{definition}\label{def:arch}
The cone $\M_x(g)$ is restricted Archimedean if $1$ is an order unit for $\M_x(g)$ in $\V_r$, that is,
\begin{equation}\label{eq:arch}
 \text{for every }p\in\V_r\text{ there is }N>0
 \text{ such that }N+p,\ N-p\in\M_x(g).
\end{equation}
\end{definition}
The order-unit property \eqref{eq:arch} is distinct from the closure property
that $p+\eps\in\M_x(g)$ for all $\eps>0$ implies $p\in\M_x(g)$.
This distinction is relevant when $p$ is nonnegative but not strictly positive.

Every element of $\M_x(g)$ is nonnegative on $K$. Therefore \eqref{eq:arch} implies that $K$ is bounded, by applying it to all coordinate functions. Since $K$ is closed, it is compact. The set is allowed to be empty in the positivity theorems. Optimization results will explicitly assume $K\neq\varnothing$.

\subsection{Algebraic bounds}
The order-unit property can be checked through the following certificates:
\begin{equation}\label{eq:bounds}
 R-\norm{x}^2\in\M_x(g),\qquad
 C_k+\la_k,\ C_k-\la_k\in\M_x(g)\quad(k\in[r]),
\end{equation}
for some $R,C_1,\ldots,C_r>0$. Each membership requires a polynomial certificate of the bound. The following
lemma also applies to the enlarged cones introduced in later sections.

\begin{lemma}\label{lem:bounds}
Let $C\subseteq\V_r$ be a convex cone containing $1$ and stable under multiplication by squares from $\A$. Then $1$ is an order unit for $C$ if and only if \eqref{eq:bounds} holds with $C$ in place of $\M_x(g)$.
\end{lemma}
\begin{proof}
The forward implication follows by applying the order-unit property to $\norm{x}^2$ and $\la_k$. Conversely, $C_0\coloneqq C\cap\A$ is a quadratic module containing $R-\norm{x}^2$. For completeness, let
\[
 B\coloneqq \{p\in\A:\text{there is }N>0\text{ with }N-p^2\in C_0\}.
\]
The set $B$ is a subring of $\A$. Indeed, if $N_p-p^2,N_q-q^2\in C_0$, then
\begin{align*}
 2(N_p+N_q)-(p+q)^2
 &=2(N_p-p^2)+2(N_q-q^2)+(p-q)^2,\\
 N_pN_q-p^2q^2&=N_q(N_p-p^2)+p^2(N_q-q^2)
\end{align*}
belong to $C_0$. Constants and all coordinates belong to $B$, so $B=\A$. Moreover,
\[
 \frac{N+1}{2}\pm p=\frac12(N-p^2)+\frac12(p\pm1)^2\in C_0.
\]
Thus $C_0$ is Archimedean.

Fix $q\in\A$ and $k\in[r]$. The identity
\begin{equation}\label{eq:polar-bound}
 \frac14(q+1)^2(C_k+\la_k)+\frac14(q-1)^2(C_k-\la_k)
 =\frac{C_k}{2}(q^2+1)+q\la_k
\end{equation}
shows that its right-hand side belongs to $C$. Interchanging the signs of $\la_k$ gives the corresponding bound for $-q\la_k$. Since $C_0$ is Archimedean, there is $D$ with $D-\frac{C_k}{2}(q^2+1)\in C_0$. Hence $D\pm q\la_k\in C$. Adding these bounds and an order bound for $p_0$ proves \eqref{eq:arch} for $p_0+\sum_k\la_kp_k$.
\end{proof}

%% file: sections/03_scalar.tex
\section{Restricted-multiplier Positivstellens\"atze}\label{sec:scalar}
We now use the order-unit framework to prove the scalar positivity theorem.
The proof first represents positive functionals for one affine parameter,
then extends the representation componentwise to several parameters.

\subsection{One affine parameter}
We first consider $\V_1=\A\oplus\la\A$, which contains the analytic
argument needed for the general case.

\begin{theorem}\label{thm:scalar}
Let $f,g_1,\ldots,g_m\in\V_1$. Suppose that $\M_x(g)$ is restricted Archimedean. If $f>0$ on $K$, then
\[
 f=\sigma_0+\sum_{i=1}^m\sigma_i g_i
 \quad\text{for some }\sigma_0,\ldots,\sigma_m\in\sos.
\]
\end{theorem}

The following order-unit separation lemma follows from
\cite[Corollary~III.1.7]{Barvinok2002} by normalization.
\begin{lemma}\label{lem:sep}
Let $C$ be a convex cone in a real vector space $V$, and let $e\in C$ be an order unit. If $v\notin C$, then there is a linear functional $\mathcal L:V\to\R$ satisfying
\[
 \mathcal L(C)\subseteq[0,\infty),\qquad \mathcal L(e)=1,\qquad \mathcal L(v)\leq0.
\]
\end{lemma}
\begin{proof}
Let $e$ be an order unit for a convex cone $C\subseteq V$. For every $w\in V$, an order bound $Ne\pm w\in C$ implies $e+tw\in C$ when $|t|<1/N$. Thus $e$ belongs to the algebraic interior, or core, of $C$. Moreover, $c+\eps e$ belongs to the core for every $c\in C$ and $\eps>0$.

If $v\notin C$, then $v$ is outside the core. The algebraic separation theorem (equivalently, Hahn--Banach separation using the gauge of an absorbing convex neighborhood in the affine span) provides a nonzero linear functional $\ell$ such that
\[
 \ell(v)\leq\ell(u)\qquad\text{for every }u\in\operatorname{core}C.
\]
Positive scaling preserves the core. Taking arbitrarily large and arbitrarily small positive multiples of a core element shows, respectively, that $\ell$ is nonnegative on the core and $\ell(v)\leq0$. Since $c+\eps e$ is a core element, letting $\eps\downarrow0$ gives $\ell(c)\geq0$ for all $c\in C$. Finally, if $\ell(e)=0$, the inequalities $Ne\pm w\in C$ imply $\ell(w)=0$ for all $w$, contradicting nonzero separation. Hence $\ell(e)>0$, and $\mathcal L\coloneqq \ell/\ell(e)$ has the desired properties.
\end{proof}

\begin{lemma}\label{lem:measure}
Let $\mathcal L:\V_1\to\R$ satisfy $\mathcal L(1)=1$ and $\mathcal L(\M_x(g))\subseteq[0,\infty)$. If $R-\norm{x}^2\in\M_x(g)$, there is a probability measure $\mu$, supported on $\{x:\norm{x}^2\leq R\}$, such that
\begin{equation}\label{eq:L0}
 \mathcal L_0(a)\coloneqq \mathcal L(a)=\int a\,d\mu\qquad(a\in\A).
\end{equation}
\end{lemma}
\begin{proof}
Since $\mathcal L_0(a^2)\geq0$, the form $\ip{a}{b}\coloneqq \mathcal L_0(ab)$ is PSD and satisfies the Cauchy--Schwarz inequality
\[
 |\mathcal L_0(ab)|^2\leq \mathcal L_0(a^2)\mathcal L_0(b^2).
\]
Let $\mathcal N\coloneqq \{a:\mathcal L_0(a^2)=0\}$. If $a\in\mathcal N$, then Cauchy--Schwarz gives $\mathcal L_0(ab)=0$ for every $b\in\A$. Consequently, $\mathcal N$ is a vector subspace, $\mathcal L_0(ab)$ depends only on the classes $[a]$ and $[b]$, and the induced form on $\A/\mathcal N$ is positive definite.

For each $j$ and $a\in\A$, positivity on $a^2(R-\norm{x}^2)$ gives
\begin{equation}\label{eq:coordbound}
 \sum_{j=1}^n \mathcal L_0(x_j^2a^2)\leq R \mathcal L_0(a^2).
\end{equation}
Equip $\A/\mathcal N$ with the inner product $\ip{[a]}{[b]}\coloneqq \mathcal L_0(ab)$ and the induced norm $\norm{[a]}\coloneqq \sqrt{\mathcal L_0(a^2)}$. Its Hilbert completion $\mathcal H$ is the complete inner-product space obtained by adjoining limits of all Cauchy sequences in this norm; thus $\A/\mathcal N$ embeds isometrically as a dense subspace of $\mathcal H$.
In particular, $a\in\mathcal N$ implies $x_ja\in\mathcal N$, so multiplication by $x_j$ is well defined on $\A/\mathcal N$. By \eqref{eq:coordbound}, it extends to a bounded self-adjoint operator $X_j$, of norm at most $\sqrt R$, on $\mathcal H$. The operators commute because they commute on the dense polynomial subspace. The vector $[1]$ is cyclic because $p(X_1,\ldots,X_n)[1]=[p]$ and $\A/\mathcal N$ is dense in $\mathcal H$. The joint spectral theorem
\cite[Theorems~5.21 and~5.23]{schmudgen2012unbounded}, applied to $[1]$,
gives a probability measure $\mu$ representing $\mathcal L_0$. The operator inequality $\sum_jX_j^2\preceq RI$ implies the asserted support bound.
\end{proof}

Let $L^2(\mu)$ denote the Hilbert space of square-integrable functions with respect to $\mu$, modulo equality almost everywhere. Polynomial functions are dense in $L^2(\mu)$, and $[a]\mapsto a$ identifies $\mathcal H$ with $L^2(\mu)$.
\begin{lemma}\label{lem:parameter}
Under the hypotheses of Lemma~\ref{lem:measure}, assume also that $C\pm\la\in\M_x(g)$ for some $C>0$. There is a real Borel function $\tau$ with $|\tau|\leq C$ $\mu$-almost everywhere such that
\begin{equation}\label{eq:L1}
 \mathcal L_1(a)\coloneqq \mathcal L(\la a)=\int a(x)\tau(x)\,d\mu(x)\qquad(a\in\A).
\end{equation}
\end{lemma}
\begin{proof}
The parameter bounds, multiplied by $a^2$, imply
\begin{equation}\label{eq:formbound}
 |\mathcal L_1(a^2)|\leq C \mathcal L_0(a^2)\qquad(a\in\A).
\end{equation}
Polarization and rescaling show that
\[
 |\mathcal L_1(ab)|\leq C\sqrt{\mathcal L_0(a^2)\mathcal L_0(b^2)}.
\]
Hence $B([a],[b])\coloneqq \mathcal L_1(ab)$ is well defined on the quotient and extends to a bounded symmetric bilinear form on $\mathcal H$. There is a bounded self-adjoint operator $T$ with $\norm{T}\leq C$ such that $\mathcal L_1(ab)=\ip{T[a]}{[b]}$.

For polynomial vectors,
\[
 \ip{TX_j[a]}{[b]}=\mathcal L_1(x_jab)=\ip{T[a]}{X_j[b]}
 =\ip{X_jT[a]}{[b]}.
\]
Thus $T$ commutes with every $X_j$ and therefore commutes with every polynomial multiplication operator. Since $\mu$ has compact support, uniform density of polynomials in continuous functions implies that $T$ commutes with every continuous function multiplication operator. Regularity of $\mu$ and approximation in $L^2(\mu)$ then show that bounded Borel multiplication operators are strong-operator limits of continuous ones. Hence $T$ commutes with every indicator multiplier $M_{1_E}$ for a Borel set $E$.

Set $\tau\coloneqq T1$. For every Borel set $E$,
\[
 T1_E=T M_{1_E}1=M_{1_E}T1=1_E\tau.
\]
Therefore
\[
 \int_E|\tau|^2\,d\mu=\norm{T1_E}^2
 \leq\norm{T}^2\mu(E).
\]
It follows that $|\tau|\leq\norm{T}\leq C$ almost everywhere. The identity $Tu=\tau u$ holds first for simple functions and then for every $u\in L^2(\mu)$ by density. Thus $T=M_\tau$. Since $T$ is self-adjoint, $\tau$ is real almost everywhere. Taking $b=1$ proves \eqref{eq:L1}.
\end{proof}

\begin{lemma}\label{lem:localization}
With $\mu$ and $\tau$ as above, $(x,\tau(x))\in K$ for $\mu$-almost every $x$.
\end{lemma}
\begin{proof}
For every $a\in\A$ and $i\in[m]$, positivity on $a^2g_i$ implies
\[
 0\leq \mathcal L(a^2g_i)=\int a(x)^2\bigl(g_{i0}(x)+\tau(x)g_{i1}(x)\bigr)\,d\mu(x).
\]
The expression in parentheses is bounded and measurable. If a bounded real function $h$ satisfies $\int a^2h\,d\mu\geq0$ for every polynomial $a$, density of polynomials in $L^2(\mu)$ gives the same inequality for every $u\in L^2(\mu)$. Indeed, $a_j\to u$ in $L^2(\mu)$ implies $\int|a_j^2-u^2|\,d\mu\to0$. Taking $u$ to be the indicator of $\{h<0\}$ shows that $h\geq0$ almost everywhere. Applying this observation to each $g_i(x,\tau(x))$ and taking the union of the finitely many null sets yields the desired result.
\end{proof}

\begin{proof}[Proof of Theorem~\ref{thm:scalar}]
Suppose that $f\notin\M_x(g)$. Lemma~\ref{lem:sep} provides a normalized positive functional $\mathcal L$ with $\mathcal L(f)\leq0$. The bounds required by Lemmas~\ref{lem:measure} and~\ref{lem:parameter} follow from Lemma~\ref{lem:bounds}. Lemma~\ref{lem:localization} gives a probability measure $\mu$ concentrated on a feasible graph,
so $K$ is nonempty. Compactness and strict positivity then give
$\delta>0$ such that $f\geq\delta$ on $K$. It follows that
\[
 \mathcal L(f)=\int f(x,\tau(x))\,d\mu(x)\geq\delta,
\]
contradicting $\mathcal L(f)\leq0$.
\end{proof}

\subsection{Several affine parameters}
The preceding argument applies to all parameter components with respect to the same principal-variable measure.

\begin{lemma}\label{thm:graph}
Suppose that $\M_x(g)$ is restricted Archimedean. A linear functional $\mathcal L:\V_r\to\R$ satisfies
\begin{equation}\label{eq:positiveL}
 \mathcal L(1)=1,\qquad \mathcal L(\M_x(g))\subseteq[0,\infty)
\end{equation}
if and only if there exists a compactly supported Borel probability measure $\mu$ on $\R^n$ and a bounded Borel map $\tau:\supp\mu\to\R^r$ such that
\begin{equation}\label{eq:graph}
 (x,\tau(x))\in K\quad\mu\text{-a.e.},\qquad
 \mathcal L(p)=\int p(x,\tau(x))\,d\mu(x)\quad(p\in\V_r).
\end{equation}
The measure $\mu$ is determined by $\mathcal L$, and each component of $\tau$ is determined $\mu$-almost everywhere.
\end{lemma}
\begin{proof}
First assume \eqref{eq:positiveL}. The restriction $\mathcal L_0\coloneqq \mathcal L|_\A$ gives a measure $\mu$ by the proof of Lemma~\ref{lem:measure}. For each $k$, set $\mathcal L_k(a)\coloneqq \mathcal L(\la_ka)$. The bound $C_k\pm\la_k\in\M_x(g)$ gives a bounded self-adjoint operator $T_k$ on $L^2(\mu)$ with $\mathcal L_k(ab)=\ip{T_ka}{b}$. Each $T_k$ commutes with the coordinate multiplication operators. Hence, as in the proof of Lemma~\ref{lem:parameter}, $T_k=M_{\tau_k}$ for an essentially bounded real function $\tau_k$.
At this stage, $\tau_k$ is an equivalence class of measurable functions modulo equality $\mu$-almost everywhere. Choose a Borel representative of each class. Since $\norm{\tau_k}_\infty\leq C_k$, the Borel set
\[
 N_k\coloneqq \{x\in\supp\mu:|\tau_k(x)|>C_k\}
\]
is $\mu$-null. The finite union $N\coloneqq \bigcup_{k=1}^rN_k$ is also $\mu$-null. Redefine every $\tau_k$ to be zero on $N$ and set
\[
 \tau(x)\coloneqq (\tau_1(x),\ldots,\tau_r(x)).
\]
The resulting map $\tau:\supp\mu\to\R^r$ is Borel and bounded everywhere. These modifications do not change the integral formulas for the $\mathcal L_k$, because they occur only on a $\mu$-null set.

For every $i$ and $a\in\A$,
\[
 0\leq \mathcal L(a^2g_i)=\int a^2\left(g_{i0}+\sum_k\tau_kg_{ik}\right)d\mu.
\]
The density argument of Lemma~\ref{lem:localization} proves feasibility almost everywhere. Linearity gives the integral identity in \eqref{eq:graph}. Conversely, evaluating any restricted SOS certificate on a feasible graph gives a nonnegative integrand, proving \eqref{eq:positiveL}.

Compactly supported measures are determined by their polynomial moments: on a compact set containing the supports of two candidate measures, polynomials are uniformly dense in the continuous functions. Thus $\mathcal L_0$ determines $\mu$. If $\tau_k$ and $\widetilde\tau_k$ give the same $\mathcal L_k$, then their difference is orthogonal to every polynomial in $L^2(\mu)$ and hence vanishes almost everywhere.
\end{proof}

\begin{theorem}[Restricted-multiplier Positivstellensatz]\label{thm:main}
Let $f,g_1,\ldots,g_m\in\V_r$. If $\M_x(g)$ is restricted Archimedean and $f>0$ on $K$, then $f\in\M_x(g)$.
\end{theorem}
\begin{proof}
Suppose that $f\notin\M_x(g)$. Lemma~\ref{lem:sep} gives a normalized positive functional $\mathcal L$ with $\mathcal L(f)\leq0$, and Lemma~\ref{thm:graph} represents $\mathcal L$ by a probability measure $\mu$ and a Borel map $\tau$ satisfying $(x,\tau(x))\in K$ for $\mu$-almost every $x$. If $K$ were empty, this would be impossible: no point $(x,\tau(x))$ could belong to $K$, whereas $\mu$ has total mass one. Hence $K$ is nonempty. Compactness and strict positivity then give $\delta>0$ such that $f\geq\delta$ on $K$, and therefore
\[
 \mathcal L(f)=\int f(x,\tau(x))\,d\mu(x)\geq\delta,
\]
contradicting $\mathcal L(f)\leq0$.
\end{proof}

\begin{corollary}\label{cor:empty}
Under restricted Archimedeanity, $K=\varnothing$ if and only if $-1\in\M_x(g)$.
\end{corollary}
\begin{proof}
Suppose first that $K=\varnothing$. Then the assertion that $-1$ is strictly positive on $K$ is vacuously true, because there is no point of $K$ at which it could fail. Theorem~\ref{thm:main}, applied to $f=-1$, therefore gives $-1\in\M_x(g)$.
Conversely, suppose that $-1\in\M_x(g)$. Then there are $\sigma_0,\ldots,\sigma_m\in\sos$ such that
\[
 -1=\sigma_0+\sum_{i=1}^m\sigma_i g_i.
\]
If $(x,\la)\in K$, then $\sigma_i(x)\geq0$ and $g_i(x,\la)\geq0$ for every $i$. Evaluating the identity at $(x,\la)$ would therefore give
\[
 -1=\sigma_0(x)+\sum_{i=1}^m\sigma_i(x)g_i(x,\la)\geq0,
\]
which is impossible. Hence $K=\varnothing$.
\end{proof}

% \begin{remark}[Barycenters and graph measures]\label{rem:barycenter}
% The graph representation does not imply that every representing measure on $K$ is concentrated on a graph. Given a probability measure $\nu$ on $K$, disintegrate it over its $x$-marginal as $\nu_x\,d\mu(x)$. Its conditional mean $\tau(x)\coloneqq \int\la\,d\nu_x(\la)$ is feasible almost everywhere because the fibers are convex. For affine $p$, one has $\int p\,d\nu=\int p(x,\tau(x))\,d\mu(x)$. Thus only the conditional mean is visible on $\V_r$; conditional variances are not part of this moment problem.
% \end{remark}

\subsection{Equality constraints}
We next incorporate
affine equalities using arbitrary polynomial multipliers in $x$.
Let $h_j=h_{j0}+\sum_k\la_kh_{j k}\in\V_r$ and define the linear subspace
\begin{equation}\label{eq:equalspace}
 \I_x(h)\coloneqq \left\{\sum_{j=1}^t \omega_j(x)h_j(x,\la):\omega_j\in\A\right\}.
\end{equation}
This is an $\A$-submodule of $\V_r$, rather than an ideal of $\R[x,\la]$. Set $K_{g,h}\coloneqq \{g_i\geq0,\ h_j=0\}$.

\begin{theorem}\label{thm:equality}
If $1$ is an order unit for $\M_x(g)+\I_x(h)$ and $f\in\V_r$ is strictly positive on $K_{g,h}$, then
\begin{equation}\label{eq:equalcert}
 f=\sigma_0+\sum_i\sigma_i g_i+\sum_j \omega_j h_j,
 \qquad \sigma_i\in\sos,\quad \omega_j\in\A.
\end{equation}
The order-unit hypothesis is equivalent to \eqref{eq:bounds} with $\M_x(g)+\I_x(h)$ in place of $\M_x(g)$.
\end{theorem}
\begin{proof}
The equivalence follows from Lemma~\ref{lem:bounds}. A normalized functional nonnegative on $\M_x(g)+\I_x(h)$ vanishes on $\I_x(h)$ because both signs of every element of this subspace belong to $\M_x(g)+\I_x(h)$. The proofs of the measure and parameter representation lemmas apply to $\M_x(g)+\I_x(h)$, since it contains the required square multiples of its bounds. They produce $\mu$ and $\tau$ with $g_i(x,\tau(x))\geq0$ almost everywhere. For every $a\in\A$,
\[
 0=\mathcal L(ah_j)=\int a(x)h_j(x,\tau(x))\,d\mu(x).
\]
The second factor is bounded. Polynomial density in $L^2(\mu)$ implies $h_j(x,\tau(x))=0$ almost everywhere. Separation and strict positivity now give the result.
\end{proof}

\begin{remark}
If there are no principal variables, $\sos=\R_+$ and $\M_x(g)$ consists of nonnegative linear combinations of $1,g_1,\ldots,g_m$. For a nonempty polyhedron, the affine Farkas lemma
\cite[Corollary~22.3.1]{Rockafellar1970} represents every affine function nonnegative on that polyhedron in this cone. On a compact polyhedron, this includes Theorem~\ref{thm:main} as a special case. The classical affine result also covers noncompact polyhedra and nonnegative objectives; those stronger conclusions use finite-dimensional polyhedral structure.
\end{remark}

%% file: sections/04_moment.tex
\section{The full restricted moment problem}\label{sec:moment}
The graph representation in Lemma~\ref{thm:graph} gives an exact solution of the moment problem naturally associated with $\M_x(g)$.  Since $\V_r$ consists of polynomials affine in $\la$, its dual is described
by one ordinary moment sequence and $r$ mixed moment sequences.
Let
\[
 y=(y_\alpha)_{\alpha\in\N^n}\in\mathbb R^{\N^n},\qquad
 z^{(k)}=(z_\alpha^{(k)})_{\alpha\in\N^n}\in\mathbb R^{\N^n},\quad k\in[r].
\]
A probability measure $\rho$ on $K$ represents $y$ and $\{z^{(k)}\}_{k\in[r]}$ if
\begin{equation}\label{eq:full-moments-intended}
 y_\alpha=\int_Kx^\alpha\,d\rho(x,\la),\qquad
 z_\alpha^{(k)}=\int_Kx^\alpha\la_k\,d\rho(x,\la).
\end{equation}
Define the associated linear functional $\mathcal L_{y,z}:\V_r\to\R$ by
\begin{equation}\label{eq:riesz-full}
 \mathcal L_{y,z}\left(p_0+\sum_{k=1}^r\la_kp_k\right)
 \coloneqq \sum_\alpha(p_0)_\alpha y_\alpha
   +\sum_{k=1}^r\sum_\alpha(p_k)_\alpha z_\alpha^{(k)}.
\end{equation}
The infinite moment matrix and the restricted localizing matrices are
\begin{align}
 M(y)&\coloneqq \bigl(y_{\alpha+\beta}\bigr)_{\alpha,\beta\in\N^n},
 \label{eq:infinite-moment-matrix}\\
 M(g_i;y,z)&\coloneqq M(g_{i0}y)+\sum_{k=1}^rM(g_{ik}z^{(k)}).
 \label{eq:infinite-localizing-matrix}
\end{align}
Thus, if $g_{ik}=\sum_\gamma g_{ik,\gamma}x^\gamma$, then
\begin{equation}\label{eq:infinite-localizing-entry}
 [M(g_i;y,z)]_{\alpha,\beta}
 =\sum_\gamma g_{i0,\gamma}y_{\alpha+\beta+\gamma}
  +\sum_{k=1}^r\sum_\gamma
       g_{ik,\gamma}z_{\alpha+\beta+\gamma}^{(k)}.
\end{equation}
An infinite symmetric matrix is PSD when each of its finite
principal submatrices is PSD.  With this convention,
\begin{align*}
 M(y)\succeq0
 &\quad\Longleftrightarrow\quad
 \mathcal L_{y,z}(q^2)\geq0 &&(q\in\R[x]),\\
 M(g_i;y,z)\succeq0
 &\quad\Longleftrightarrow\quad
 \mathcal L_{y,z}(q^2g_i)\geq0 &&(q\in\R[x]).
\end{align*}

These matrix inequalities characterize representability under the
restricted Archimedean hypothesis.

\begin{theorem}\label{thm:full-moment}
Suppose that $\M_x(g)$ is restricted Archimedean.  For pseudo-moment sequences
$y,z^{(1)},\ldots,z^{(r)}$, the following statements are equivalent:
\begin{enumerate}[label=(\roman*)]
\item There is a Borel probability measure $\rho$ supported on $K$ that
satisfies \eqref{eq:full-moments-intended} for every $\alpha\in\N^n$ and
$k\in[r]$.
\item $y_0=1$ and $\mathcal L_{y,z}$ is nonnegative on $\M_x(g)$.
\item
\[
 y_0=1,\qquad M(y)\succeq0,\qquad
 M(g_i;y,z)\succeq0\quad(i\in[m]).
\]
\end{enumerate}
Whenever these conditions hold, there is a compactly supported probability
measure $\mu$ on $\R^n$ and functions
$\tau_k\in L^\infty(\mu)$ such that
\begin{equation}\label{eq:full-graph-moments}
 y_\alpha=\int x^\alpha\,d\mu(x),\qquad
 z_\alpha^{(k)}=\int x^\alpha\tau_k(x)\,d\mu(x),
\end{equation}
and $(x,\tau(x))\in K$ for $\mu$-almost every $x$.  If the bounds in
\eqref{eq:bounds} hold, the representatives may be chosen so that
$|\tau_k|\leq C_k$ almost everywhere.  In particular, the pushforward
\begin{equation}\label{eq:canonical-graph-measure}
 \rho\coloneqq \bigl(x\mapsto(x,\tau(x))\bigr)_\#\mu
\end{equation}
is a representing measure.
\end{theorem}
\begin{proof}
The equivalence of (ii) and (iii) follows from
\eqref{eq:infinite-moment-matrix}--\eqref{eq:infinite-localizing-entry} and the
definition of $\M_x(g)$.  Condition (i) implies (ii) by integration.  If
(ii) holds, Theorem~\ref{thm:graph}, applied to $\mathcal L_{y,z}$, gives $\mu$ and
$\tau$ satisfying \eqref{eq:full-graph-moments} and the almost-everywhere
feasibility statement.  The parameter bounds follow from
Lemma~\ref{lem:parameter}, component by component.  Finally,
\eqref{eq:canonical-graph-measure} satisfies
\eqref{eq:full-moments-intended}, so (i) follows.
\end{proof}

% No moments containing products $\la_j\la_k$ or higher powers of $\la$
% are needed: the data in \eqref{eq:riesz-full} specify every linear functional
% on $\V_r$.

The canonical graph measure need not be the only measure with the prescribed
restricted moments.  The next result characterizes this nonuniqueness.
For $x\in\R^n$, write
\[
 \Lambda(x)\coloneqq \{\la\in\R^r:g_i(x,\la)\geq0,\ i\in[m]\}.
\]
Every $\Lambda(x)$ is a closed convex set.

\begin{proposition}\label{prop:all-representing-measures}
Assume the equivalent conditions of Theorem~\ref{thm:full-moment}, and let
$(\mu,\tau)$ be its uniquely determined graph data, up to $\mu$-null sets.
A probability measure $\rho$ on $K$ represents $(y,z)$ if and only if it has
a disintegration
\begin{equation}\label{eq:disintegration}
 \rho(dx,d\la)=\mu(dx)\,\pi_x(d\la)
\end{equation}
such that, for $\mu$-almost every $x$,
\begin{equation}\label{eq:kernel-conditions}
 \supp\pi_x\subseteq \Lambda(x),\qquad
 \int_{\Lambda(x)}\la\,d\pi_x(\la)=\tau(x).
\end{equation}
Thus all representing measures have the same $x$-marginal and the same
conditional parameter barycenter.
\end{proposition}
\begin{proof}
Let $\rho$ be a representing measure.  Restricted Archimedeanity makes $K$
compact.  Disintegration over the $x$-marginal therefore gives
\eqref{eq:disintegration} with a measurable probability kernel $\pi_x$.
The $y$-moments determine the compactly supported $x$-marginal uniquely by
the Stone--Weierstrass theorem, so this marginal is $\mu$.  Put
$\widetilde\tau(x)\coloneqq \int\la\,d\pi_x(\la)$.  The support condition gives
$\supp\pi_x\subseteq \Lambda(x)$.  Since $\Lambda(x)$ is convex,
$\widetilde\tau(x)\in \Lambda(x)$ almost everywhere.  Equality of the $z$-moments
implies
\[
 \int x^\alpha(\widetilde\tau_k-\tau_k)\,d\mu=0
 \qquad(\alpha\in\N^n).
\]
Polynomial density in $L^2(\mu)$ yields
$\widetilde\tau_k=\tau_k$ almost everywhere.  Conversely, any measurable
kernel satisfying \eqref{eq:kernel-conditions} reproduces
\eqref{eq:full-graph-moments} after integration.
\end{proof}

Therefore, replacing $\pi_x$ by $\delta_{\tau(x)}$ preserves feasibility and
all prescribed restricted moments. The remaining freedom lies in the higher
conditional moments of $\la$.

%% file: sections/05_hierarchy.tex
\section{A restricted moment-SOS hierarchy}\label{sec:hierarchy}
Throughout this section, assume $K\neq\varnothing$ and consider the polynomial optimization problem
\begin{equation}\label{eq:opt}
f_{\min}\coloneqq \min_{(x,\la)\in K}f(x,\la).
\end{equation}
The restricted-multiplier Positivstellensatz developed above leads naturally to moment-SOS relaxations tailored to \eqref{eq:opt} with smaller SDP sizes.

For
$p=p_0+\sum_k\la_kp_k\in\V_r$, define
\[
 \deg_xp\coloneqq \max_{0\leq k\leq r}\deg p_k.
\]
Set $g_0\coloneqq 1$ and $d_i\coloneqq \lceil \deg_xg_i/2\rceil$ for $i=0,\ldots,m$. For any relaxation order
\[
 d\geq d_{\min}\coloneqq 
 \max\{\lceil\deg_x f/2\rceil,d_1,\ldots,d_m\},
\]
define the truncated restricted quadratic module
\begin{equation}\label{eq:xdegree-cone}
 \M_{x,d}(g)\coloneqq 
 \left\{\sigma_0+\sum_{i=1}^m\sigma_i g_i:
 \begin{array}{l}
 \sigma_0\in\Sigma[x]_{\leq2d},\\[-1mm]
 \sigma_i\in\Sigma[x]_{\leq2(d-d_i)}
 \end{array}\right\},
\end{equation}
and let
\begin{equation}
f_d\coloneqq \sup\{\gamma\in\R:f-\gamma\in\M_{x,d}(g)\}.
 \label{eq:soshierarchy}
\end{equation}

\begin{theorem}\label{thm:convergence}
The bounds satisfy $f_d\leq f_{d+1}\leq f_{\min}$.
Moreover, if $\M_x(g)$ is restricted Archimedean, then $\lim_{d\to\infty}f_d=f_{\min}$.
\end{theorem}
\begin{proof}
Nestedness gives monotonicity, and evaluation on $K$ gives the upper bound.
For $\eps>0$, the polynomial $f-f_{\min}+\eps$ is strictly positive on $K$.
Theorem~\ref{thm:main} supplies a restricted certificate of finite degree, which belongs
to $\M_{x,d}(g)$ for all sufficiently large $d$.  Hence
$f_d\geq f_{\min}-\eps$ eventually.
\end{proof}

% Writing
% $\sigma_i=[x]_{d-d_i}^\intercalQ_i[x]_{d-d_i}$ with $Q_i\succeq0$ and matching
% coefficients of $1,\la_1,\ldots,\la_r$ gives
% \begin{subequations}\label{eq:matching}
% \begin{align}
%  f_0-\gamma&=\sigma_0+\sum_{i=1}^m\sigma_i g_{i0},\label{eq:matching0}\\
%  f_k&=\sum_{i=1}^m\sigma_i g_{ik},\qquad k\in[r].
%  \label{eq:matchingk}
% \end{align}
% \end{subequations}
% These are linear equations in $\gamma$ and the Gram entries.  No positive
% semidefinite block is indexed by parameter monomials.

To simplify notation, write $y^0\coloneqq y$ and
$y^k\coloneqq z^{(k)}$, $k\in[r]$ for the pseudo-moment sequences. Set $\mathbf y\coloneqq(y^0,\ldots,y^r)$.
For $k=0,\ldots,r$, define the $d$-th order moment matrix
\begin{equation}\label{eq:momentmatrix}
 M_d(y^k)\coloneqq (y^k_{\alpha+\beta})_{|\alpha|,|\beta|\leq d}.
\end{equation}
Let $g_{ik}(x)=\sum_\gamma g_{ik,\gamma}x^\gamma$, and define the restricted localizing matrix
\begin{equation}\label{eq:localmatrix}
 M_{d-d_i}(g_i\mathbf y)\coloneqq 
 \left(\sum_{k=0}^r\sum_\gamma g_{ik,\gamma}
        y^k_{\alpha+\beta+\gamma}\right)_{|\alpha|,|\beta|\leq d-d_i}.
\end{equation}
The conic dual of \eqref{eq:soshierarchy} reads
\begin{equation}\label{eq:momentdual}
f_d^*\coloneqq\inf_\mathbf y\left\{\mathcal L_\mathbf y(f):y^0_0=1,\ M_d(y^0)\succeq0,
       \ M_{d-d_i}(g_i\mathbf y)\succeq0\ (i\in[m])\right\}.
\end{equation}

\begin{remark}
The largest PSD block in the restricted relaxation
\eqref{eq:soshierarchy}--\eqref{eq:momentdual} has size
$\binom{n+d}{d}$, because its Gram and moment matrices are indexed only by
monomials in the principal variables. In contrast, the corresponding block
in the standard joint-variable relaxation has size
$\binom{n+r+d}{d}$, since monomials in both $x$ and $\la$ are included. The
resulting reduction becomes increasingly pronounced as either the number of
parameters $r$ or the relaxation order $d$ grows.
\end{remark}

% \begin{proposition}\label{prop:dualconvergence}
% For every admissible $d$,
% \[
% f_d\leq\gamma_d\leq f_{\min}.
% \]
% \end{proposition}
% \begin{proof}
% The first inequality is weak conic duality.  Evaluation at a minimizer
% $(x^*,\la^*)$ gives a feasible sequence
% $y^0_\alpha=(x^*)^\alpha$ and
% $y^k_\alpha=\la_k^*(x^*)^\alpha$ of value $f_{\min}$.
% Truncating a feasible order-$(d+1)$ sequence gives an order-$d$ feasible
% sequence, hence $\gamma_d\leq\gamma_{d+1}$.  Theorem~\ref{thm:convergence} and
% the two-sided inequalities prove convergence.
% \end{proof}

% \begin{proposition}\label{prop:slater}
% Suppose the SOS relaxation is feasible and has finite optimal value.  If
% \eqref{eq:momentdual} has a feasible point at which every displayed moment
% and localizing matrix is positive definite, then $f_d=\gamma_d$, and the
% SOS supremum is attained.
% \end{proposition}
% \begin{proof}
% This is finite-dimensional semidefinite programming duality under strict
% feasibility of the dual \cite{VandenbergheBoyd1996}.
% \end{proof}
% Restricted Archimedeanity alone does not imply strict feasibility; lower-dimensional feasible sets and redundant constraints may prevent it.

We state a sufficient, directly checkable exactness criterion for
\eqref{eq:momentdual}.  It combines flatness of the ordinary moment matrix with domination conditions for the mixed moment matrices.

\begin{theorem}[Global optimality]\label{thm:flat-extraction}
Let $\mathbf y$ be an optimal solution of \eqref{eq:momentdual}.  Suppose, in addition, that for some $C_k>0$,
\begin{equation}\label{eq:parameter-domination}
 -C_kM_d(y^0)\preceq M_d(y^k)\preceq C_kM_d(y^0),
 \qquad k\in[r].
\end{equation}
Let $d_K\coloneqq \max\{1,d_1,\ldots,d_m\}$. If
\begin{equation}\label{eq:flatness}
 \rank M_d(y^0)=\rank M_{d-d_K}(y^0)=s,
\end{equation}
then $f_d^*=f_{\min}$. Moreover, there exist distinct points
$(x^{(j)},\la^{(j)})\in K$ and positive weights $w_j$,
$j=1,\ldots,s$, with $\sum_jw_j=1$, such that the finitely atomic measure $\rho=\sum_{j=1}^sw_j\delta_{(x^{(j)},\la^{(j)})}$ represents $\mathbf y$ and each $(x^{(j)},\la^{(j)})$ is a global minimizer of \eqref{eq:opt}.
\end{theorem}
\begin{proof}
The flat-extension theorem for the ordinary moment matrix
\cite{CurtoFialkow1998} gives distinct points $x^{(1)},\ldots,x^{(s)}$ and
positive weights $w_j$ such that $\mu=\sum_{j=1}^sw_j\delta_{x^{(j)}}$ represents $y^0$.  Let
$\mathcal H$ be the finite-dimensional flat GNS space.  It is spanned by
polynomials of degree at most $d-d_K$, and multiplication by the coordinate
functions gives commuting self-adjoint operators $X_1,\ldots,X_n$ whose joint
eigenvalues are the recovered points.

For each $k$, the domination condition \eqref{eq:parameter-domination} implies
$\ker M_d(y^0)\subseteq\ker M_d(y^k)$.  Hence
\[
 \ip{T_k[p]}{[q]}\coloneqq \mathcal L_{y^k}(pq),\qquad p,q\in\R[x]_{\leq d},
\]
defines a self-adjoint operator $T_k$ on $\mathcal H$ with
$-C_kI\preceq T_k\preceq C_kI$. By flatness, every vector in $\mathcal H$
has a representative of degree at most $d-d_K$. Choose such representatives
$p,q$. Since $d_K\geq1$, both $x_i p$ and $x_i q$ have degree at most
$d$, so all expressions below involve available moments. Using the Hankel
identity, which allows the factor $x_i$ to be moved between the two
arguments, and the self-adjointness of $X_i$, we obtain
\begin{align*}
 \ip{T_kX_i[p]}{[q]}
 &=\mathcal L_{y^k}((x_i p)q)
  =\mathcal L_{y^k}(p(x_i q))\\
 &=\ip{T_k[p]}{X_i[q]}
  =\ip{X_i T_k[p]}{[q]}.
\end{align*}
Since the classes of these low-degree polynomials span $\mathcal H$, this
proves $T_kX_i=X_i T_k$ for every $i$. The joint eigenspaces of the
$X_i$ are the one-dimensional
atomic spaces.  Thus $T_k$ is diagonal in the atomic basis; write its value on
the $j$th space as $\la_k^{(j)}$. To recover the mixed moments, fix
$\alpha\in\N^n$ with $|\alpha|\leq2d$ and choose multi-indices
$\beta,\gamma$ such that
\[
 \alpha=\beta+\gamma,\qquad |\beta|\leq d,\qquad |\gamma|\leq d.
\]
Then, by the definition of $T_k$ and the atomic representation of the inner
product,
\begin{align*}
 y^k_\alpha
 &=\mathcal L_{y^k}(x^\alpha)
  =\mathcal L_{y^k}(x^\beta x^\gamma)
  =\ip{T_k[x^\beta]}{[x^\gamma]}\\
 &=\sum_{j=1}^s w_j\la_k^{(j)}
      (x^{(j)})^\beta(x^{(j)})^\gamma
  =\sum_{j=1}^s w_j\la_k^{(j)}(x^{(j)})^\alpha.
\end{align*}

The rank equality implies that the evaluation map from
$\R[x]_{\leq d-d_K}$ onto the functions on the $s$ points is surjective.
Choose interpolation polynomials $\ell_j$ in this space with
$\ell_j(x^{(q)})=\delta_{jq}$.  Since
$d-d_K\leq d-d_i$, the $i$th localizing constraint gives
\[
 0\leq \mathcal L_\mathbf y(\ell_j^2g_i)
   =w_jg_i(x^{(j)},\la^{(j)}).
\]
Therefore every recovered pair belongs to $K$.

The objective value is the integral of $f$ against the finitely atomic measure $\rho$.  It is therefore at least $f_{\min}$.
Evaluation at a global minimizer gives the reverse inequality for every
moment relaxation.  Hence equality holds.  A positive weighted average of
numbers bounded below by $f_{\min}$ can equal $f_{\min}$ only if every term with
positive weight equals $f_{\min}$.
\end{proof}

\begin{remark}
The domination conditions \eqref{eq:parameter-domination} are among the localizing constraints when
$C_k\pm\la_k$ are included as generators.  
\end{remark}

%% file: sections/06_extensions.tex
\section{Extensions to polynomial-matrix settings}\label{sec:extensions}
We next extend the scalar theory to polynomial matrices. Scalar objectives
retain the graph representation under matrix constraints. For matrix-valued
objectives, one parameter admits a commutative representation, whereas
several parameters require a stronger positivity hypothesis in general.

\subsection{Scalar objectives with polynomial-matrix constraints}
For $i\in[m]$, let
\[
 G_i(x,\la)=G_{i0}(x)+\sum_{k=1}^r\la_kG_{ik}(x)
 \in\Sym^{q_i}[x,\la].
\]
Write $\Sigma^{q}[x]$ for the cone of SOS polynomial matrices of size $q$.
Thus $S\in\Sigma^q[x]$ if $S=P^{\intercal}P$ for some polynomial matrix $P$ with
$q$ columns.  Define
\begin{align}
 \Q_x(G)&\coloneqq \left\{\sigma_0+\sum_{i=1}^m\ip{S_i}{G_i}:
       \sigma_0\in\Sigma[x],\ S_i\in\Sigma^{q_i}[x]\right\},
 \label{eq:matrixcone}\\
 K_G&\coloneqq \{(x,\la):G_i(x,\la)\succeq0,\ i\in[m]\}.
\end{align}
Here $\ip{S_i}{G_i}=\tr(S_iG_i)$ is scalar and affine in $\la$.

\begin{theorem}\label{thm:matrixconstraints}
Suppose that $1$ is an order unit for $\Q_x(G)\subseteq\V_r$.  If a scalar
$f\in\V_r$ is strictly positive on $K_G$, then
\begin{equation}\label{eq:matrixcert}
 f=\sigma_0+\sum_{i=1}^m\ip{S_i(x)}{G_i(x,\la)},
 \qquad \sigma_0\in\Sigma[x],\quad S_i\in\Sigma^{q_i}[x].
\end{equation}
The order-unit condition is equivalent to the ball and parameter bounds in
\eqref{eq:bounds}, with $\Q_x(G)$ in place of $\M_x(g)$.
\end{theorem}
\begin{proof}
Suppose $f\notin\Q_x(G)$. By Lemma~\ref{lem:sep}, choose a normalized separating functional
$\mathcal L$. The cone $\Q_x(G)$ is convex, contains $1$,
and is stable under multiplication by squares in $x$. Lemma~\ref{lem:bounds}
and the representation argument in Lemma~\ref{thm:graph} therefore give a
compactly supported probability measure $\mu$ and a bounded measurable map
$\tau$ representing $\mathcal L$.
For a constant vector $u\in\mathbb Q^{q_i}$ and $a\in\R[x]$,
\[
 a^2u^{\intercal}G_i u=\ip{(au)(au)^{\intercal}}{G_i}\in\Q_x(G).
\]
For fixed $u$, set $h_u(x)\coloneqq u^{\intercal}G_i(x,\tau(x))u$. Positivity of the separating
functional gives
\[
 \int a(x)^2h_u(x)\,d\mu(x)\geq0
 \qquad\text{for every }a\in\R[x].
\]
Polynomial density in $L^2(\mu)$,
as in Lemma~\ref{lem:localization}, implies
$h_u(x)\geq0$ almost everywhere. Taking the union of the null
sets over the countable set $\mathbb Q^{q_i}$ and using density gives
$G_i(x,\tau(x))\succeq0$ almost everywhere.  Separation now contradicts
strict positivity of $f$, exactly as in Theorem~\ref{thm:main}.
\end{proof}

% The corresponding SOS hierarchy replaces each $S_i$ by an SOS-matrix Gram
% representation in the monomial basis of $x$.  Its dual localizing block is
% indexed by pairs consisting of a matrix coordinate and an $x$-monomial; no
% parameter monomial enters the block index.  The proof of
% Theorem~\ref{thm:convergence} therefore gives asymptotic convergence for a
% scalar objective over $K_G$.

\subsection{Matrix objectives and one affine parameter}
Let
\[
 F(x,\la)=F_0(x)+\sum_{k=1}^r\la_kF_k(x)
 \in\Sym^p[x,\la].
\]
For the matrix constraints above, define the restricted matrix quadratic module
\begin{equation}\label{eq:lifted-matrix-module}
 \mathfrak M_x^p(G)\coloneqq 
 \left\{S_0+\sum_{i=1}^m\sum_\nu
 H_{i\nu}^{\intercal}G_iH_{i\nu}:
 \begin{array}{l}
 S_0\in\Sigma^p[x],\\[-1mm]
 H_{i\nu}\in\R[x]^{q_i\times p}
 \end{array}\right\}.
\end{equation}
For scalar constraints $g_i$, this reduces to
$\Sigma^p[x]+\sum_i g_i\Sigma^p[x]$.
We say that $\mathfrak M_x^p(G)$ is restricted Archimedean if $I_p$ is an order unit in
$\Sym^p\otimes\V_r$.

We first record the amplified GNS construction used in both the one- and
several-parameter arguments.  If
$P=P_0+\sum_k\la_kP_k\in\Sym^s\otimes\V_r$, write
\[
 P(X,T)\coloneqq
 \left(P_{ab,0}(X)+\sum_{k=1}^rP_{ab,k}(X)T_k\right)_{a,b=1}^s,
\]
where $P_{ab}=P_{ab,0}+\sum_k\la_kP_{ab,k}$.  This is well defined whenever
every $T_k$ commutes with the operators $X_j$; no commutativity among the
$T_k$ is required because $P$ is affine in $\la$.

\begin{lemma}[Amplified GNS and localization]\label{lem:matrix-gns}
Suppose that $\mathfrak M_x^p(G)$ is restricted Archimedean and that
$\mathcal L:\Sym^p\otimes\V_r\to\R$ satisfies
\[
 \mathcal L(I_p)=1,\qquad
 \mathcal L(\mathfrak M_x^p(G))\subseteq[0,\infty).
\]
Then there is a separable real Hilbert space $\mathcal H$, commuting bounded
self-adjoint operators $X_1,\ldots,X_n$, bounded self-adjoint operators
$T_1,\ldots,T_r$ commuting with every $X_j$, and vectors
$\xi_1,\ldots,\xi_p\in\mathcal H$ such that
\begin{align}
 \mathcal L(P)
 &=\sum_{a,b=1}^p
   \ip{P_{ab}(X,T)\xi_b}{\xi_a}
 &&(P\in\Sym^p\otimes\V_r),
 \label{eq:matrix-gns-functional}\\
 \sum_{a=1}^p\norm{\xi_a}^2&=1,
 \label{eq:matrix-gns-normalization}\\
 G_i(X,T)&\succeq0\quad\text{on }\mathcal H^{q_i}
 &&(i\in[m]).
 \label{eq:matrix-gns-localization}
\end{align}
Moreover, certified bounds
$RI_p-\norm{x}^2I_p\in\mathfrak M_x^p(G)$ and
$C_kI_p\pm\la_kI_p\in\mathfrak M_x^p(G)$ imply
$\sum_jX_j^2\preceq RI$ and $-C_kI\preceq T_k\preceq C_kI$.
\end{lemma}
\begin{proof}
The matrix module $\mathfrak M_x^p(G)$ is closed under polynomial matrix congruence.  Restricted
Archimedeanity therefore supplies numbers $R,C_k>0$ for which the displayed
ball and parameter bounds hold and remain positive after every such
congruence.

On the vector-polynomial space $\mathcal D\coloneqq\R[x]^p$, define
\[
 \ip{u}{v}_0\coloneqq
 \mathcal L\!\left(\frac{uv^{\intercal}+vu^{\intercal}}{2}\right).
\]
This form is PSD because $uu^{\intercal}$ is an SOS polynomial matrix.  Let
$\mathcal N\coloneqq\{u:\ip{u}{u}_0=0\}$ and let $\mathcal H$ be the
Hilbert completion of $\mathcal D/\mathcal N$. Applying a polynomial-matrix congruence to the ball bound
with a polynomial matrix whose first row is $u^{\intercal}$ gives
\[
 \sum_{j=1}^n\norm{[x_ju]}^2\leq R\norm{[u]}^2.
\]
Thus multiplication by $x_j$ descends to a bounded self-adjoint operator
$X_j$ on $\mathcal H$.  These operators commute on the dense polynomial
subspace and hence commute everywhere.

For each $k$, consider the symmetric bilinear form
\[
 B_k([u],[v])\coloneqq
 \mathcal L\!\left(
   \la_k\frac{uv^{\intercal}+vu^{\intercal}}{2}\right).
\]
Applying the same congruence to $C_kI_p\pm\la_kI_p$ gives
$|B_k([u],[u])|\leq C_k\norm{[u]}^2$.  Polarization and rescaling, as in
Lemma~\ref{lem:parameter}, yield
\[
 |B_k([u],[v])|\leq C_k\norm{[u]}\norm{[v]}.
\]
Hence the form descends to the quotient and is represented by a bounded
self-adjoint operator $T_k$ satisfying
$-C_kI\preceq T_k\preceq C_kI$.  For polynomial vectors,
\[
 \ip{T_kX_j[u]}{[v]}
 =B_k([x_ju],[v])
 =B_k([u],[x_jv])
 =\ip{X_jT_k[u]}{[v]},
\]
so $T_kX_j=X_jT_k$.

Let $e_a$ be the $a$th constant coordinate vector and put
$\xi_a\coloneqq[e_a]$.  The definitions of $X$ and $T$ give, for every
scalar polynomial $q$,
\begin{align*}
 \ip{q(X)\xi_b}{\xi_a}
 &=\mathcal L\!\left(
   q\frac{e_be_a^{\intercal}+e_ae_b^{\intercal}}{2}\right),\\
 \ip{q(X)T_k\xi_b}{\xi_a}
 &=\mathcal L\!\left(
   q\la_k\frac{e_be_a^{\intercal}+e_ae_b^{\intercal}}{2}\right).
\end{align*}
Summing over the entries of a symmetric matrix polynomial proves
\eqref{eq:matrix-gns-functional}, and
\eqref{eq:matrix-gns-normalization} follows by taking $P=I_p$.

It remains to prove the amplified constraint inequalities.  Write
$G_i=(g_{i,ab})_{a,b=1}^{q_i}$ and take arbitrary
$u_1,\ldots,u_{q_i}\in\mathcal D$.  Let
$U\in\R[x]^{q_i\times p}$ have $a$th row $u_a^{\intercal}$.  Using the symmetry of
$G_i$ and the preceding identities,
\[
 \sum_{a,b=1}^{q_i}
 \ip{g_{i,ab}(X,T)[u_b]}{[u_a]}
 =\mathcal L(U^{\intercal}G_iU)\geq0.
\]
The last inequality holds because $U^{\intercal}G_iU$ is a generator term in
$\mathfrak M_x^p(G)$.  Since polynomial vectors are dense in
$\mathcal H^{q_i}$, this proves \eqref{eq:matrix-gns-localization}.
\end{proof}

The following representation principle is the matrix analogue of the scalar
graph construction and explains the special role of one parameter.

\begin{lemma}\label{lem:matrix-representation}
Let $r=1$, suppose that $\mathfrak M_x^p(G)$ is restricted Archimedean, and let
$\mathcal L:\Sym^p\otimes\V_1\to\R$ satisfy
\[
 \mathcal L(I_p)=1,\qquad
 \mathcal L(\mathfrak M_x^p(G))\subseteq[0,\infty).
\]
Then there is a PSD $p\times p$ matrix-valued Borel measure
$\Omega$ supported on $K_G$ such that $\tr\Omega(K_G)=1$ and
\begin{equation}\label{eq:matrix-measure-representation}
 \mathcal L(P)=\int_{K_G}\tr\bigl(P(x,\la)\,d\Omega(x,\la)\bigr)
 \qquad(P\in\Sym^p\otimes\V_1).
\end{equation}
\end{lemma}
\begin{proof}
Apply Lemma~\ref{lem:matrix-gns}.  Since $r=1$, the bounded self-adjoint
operators $X_1,\ldots,X_n,T$ commute and therefore have a joint projection-valued spectral measure $E$ on $\R^{n+1}$.  For a Borel set $B$, define
\begin{equation}\label{eq:matrix-valued-spectral-measure}
 \Omega(B)_{ab}\coloneqq\ip{E(B)\xi_a}{\xi_b},
 \qquad a,b\in[p].
\end{equation}
This is a countably additive symmetric matrix-valued measure, and for every
$c\in\R^p$,
\[
 c^{\intercal}\Omega(B)c
 =\ip{E(B)\sum_ac_a\xi_a}{\sum_bc_b\xi_b}\geq0.
\]
Thus $\Omega$ is PSD, while
\[
 \tr\Omega(\R^{n+1})
 =\sum_a\norm{\xi_a}^2=1.
\]

We next verify the support.  For $c\in\mathbb Q^{q_i}$,
\eqref{eq:matrix-gns-localization} implies
\[
 c^{\intercal}G_i(X,T)c\succeq0.
\]
The joint spectral theorem therefore gives
\[
 E\bigl(\{(x,t):c^{\intercal}G_i(x,t)c<0\}\bigr)=0.
\]
A symmetric matrix fails to be PSD if and only if its quadratic form is
negative on some rational vector.  Taking the countable union over
$i\in[m]$ and $c\in\mathbb Q^{q_i}$ shows that $E$ is concentrated on
$K_G$.  Equation \eqref{eq:matrix-valued-spectral-measure} then shows that
$\Omega$ is supported on $K_G$ as well.

Finally, for $P\in\Sym^p\otimes\V_1$, the scalar joint functional calculus
and \eqref{eq:matrix-gns-functional} give
\begin{align*}
 \int_{K_G}\tr(P\,d\Omega)
 &=\sum_{a,b=1}^p\int_{K_G}P_{ab}\,d\Omega_{ba}\\
 &=\sum_{a,b=1}^p\ip{P_{ab}(X,T)\xi_b}{\xi_a}
 =\mathcal L(P).
\end{align*}
This proves every assertion, including the coefficientwise integral
convention in \eqref{eq:matrix-measure-representation}.
\end{proof}

For $r\geq2$, Lemma~\ref{lem:matrix-gns} still produces bounded
self-adjoint operators $T_1,\ldots,T_r$ commuting with the
principal-variable operators, but it does not force $T_jT_k=T_kT_j$. Thus a
joint scalar spectral measure, and hence the representation
\eqref{eq:matrix-measure-representation}, need not exist. This is the
operator obstruction exhibited concretely in Example~\ref{ex:matrixfailure}.

\begin{theorem}[One-parameter matrix Positivstellensatz]
\label{thm:one-parameter-matrix}
Let $r=1$, and suppose that $\mathfrak M_x^p(G)$ is restricted Archimedean.  If
\[
 F(x,\la)\succ0\qquad\text{for every }(x,\la)\in K_G,
\]
then $F\in\mathfrak M_x^p(G)$.
\end{theorem}
\begin{proof}
If $F$ were outside the cone, order-unit separation would give a normalized
linear functional on $\Sym^p\otimes\V_1$ that is nonnegative on
$\mathfrak M_x^p(G)$ and nonpositive at $F$.  Lemma~\ref{lem:matrix-representation} then represents this
functional by a positive matrix-valued measure supported on $K_G$.  Since
$K_G$ is compact and $F\succ0$ there, $F\succeq\eps I_p$ on $K_G$ for some
$\eps>0$. Positivity of the matrix-valued measure gives
\[
 \mathcal L(F)\geq\eps\tr\Omega(K_G)=\eps,
\]
contradicting the separating inequality $\mathcal L(F)\leq0$.
\end{proof}

When the constraints are scalar and $\M_x(g)$ is restricted Archimedean, its
matrix lift is also restricted Archimedean.  Indeed, scalar order bounds can be multiplied by
constant rank-one PSD matrices, and the usual
$e_a\pm e_b$ decomposition bounds the off-diagonal entries of an arbitrary
symmetric polynomial matrix.

\subsection{Matrix objectives and several affine parameters}
The preceding theorem cannot be extended to several parameters under
ordinary pointwise positivity.

\begin{example}\label{ex:matrixfailure}
Fix $c\geq0$. Let the parameter set be $[-1,1]^2$, described by
$1\pm\la_1\geq0$ and $1\pm\la_2\geq0$. Put
\[
 Z=\begin{pmatrix}1&0\\0&-1\end{pmatrix},\qquad
 X=\begin{pmatrix}0&1\\1&0\end{pmatrix},\qquad
 F_c(\la)=I_2+c\la_1Z+c\la_2X.
\]
Because $Z^2=X^2=I_2$ and $ZX+XZ=0$, the eigenvalues of $F_c(\la)$ are
\[
 1\pm c\sqrt{\la_1^2+\la_2^2}.
\]
Hence $F_c\succ0$ on the square whenever $c<1/\sqrt2$.

Suppose that a restricted certificate existed:
\begin{equation}\label{eq:matriximpossible}
 F_c=S_0+(1+\la_1)A_++(1-\la_1)A_-
          +(1+\la_2)B_++(1-\la_2)B_-,
\end{equation}
where all five coefficient matrices are PSD.  Coefficient
matching gives
\[
 A_+-A_-=cZ,\qquad B_+-B_-=cX,
 \qquad S_0+A_++A_-+B_++B_-=I_2.
\]
For PSD $P,Q$,
$\tr(P+Q)\geq\|P-Q\|_*$, where $\|\cdot\|_*$ is the nuclear norm.  Therefore
\[
 \tr(A_++A_-)\geq2c,\qquad
 \tr(B_++B_-)\geq2c.
\]
Taking traces in the constant equation yields $2\geq4c$.  Thus no
certificate exists when $c>1/2$.

The threshold is exact.  If $0\leq c\leq1/2$, set
\[
 A_\pm=\frac c2(I_2\pm Z),\qquad
 B_\pm=\frac c2(I_2\pm X),\qquad
 S_0=(1-2c)I_2.
\]
These matrices are PSD and satisfy
\eqref{eq:matriximpossible}.  Consequently,
\[
 \begin{array}{ccl}
 F_c\succeq0\text{ on }[-1,1]^2&\Longleftrightarrow&c\leq1/\sqrt2,\\
 F_c\text{ has the restricted certificate}&\Longleftrightarrow&c\leq1/2.
 \end{array}
\]
The gap $1/2<c<1/\sqrt2$ disproves the several-parameter matrix extension of the restricted Positivstellensatz even in the absence of principal variables.
\end{example}

The obstruction has an operator-theoretic interpretation.  A matrix-valued separating
functional has a nontrivial multiplicity space.  Each parameter yields a
self-adjoint operator $T_k$ commuting with the principal-variable operators,
but $T_j$ and $T_k$ need not commute with one another.  Scalar evaluation on
$\R^r$ tests only commuting tuples.  In Example~\ref{ex:matrixfailure}, the
noncommuting contraction tuple $(Z,X)$ witnesses the stronger obstruction:
for $c>1/2$, the free evaluation
$I_2\otimes I_2+cZ\otimes Z+cX\otimes X$ is not PSD.

Thus, for several parameters, pointwise positivity should be replaced by the uniform free positivity condition below.
For $x\in\R^n$, a $q\times q$ symmetric matrix tuple
$A=(A_1,\ldots,A_r)$ is feasible if
\begin{equation}\label{eq:free-feasible}
 G_i(x,A)\coloneqq G_{i0}(x)\otimes I_q
       +\sum_{k=1}^rG_{ik}(x)\otimes A_k\succeq0
 \quad(i\in[m]).
\end{equation}
Define $F(x,A)$ analogously.

\begin{theorem}\label{thm:free-matrix}
Suppose that $\mathfrak M_x^p(G)$ is restricted Archimedean and that there is
$\eps>0$ such that
\begin{equation}\label{eq:uniform-free-positive}
 F(x,A)\succeq\eps I_{pq}
\end{equation}
for every $q\geq1$, every $x\in\R^n$, and every tuple $A$ satisfying
\eqref{eq:free-feasible}.  Then $F\in\mathfrak M_x^p(G)$.
\end{theorem}
\begin{proof}
Suppose $F\notin\mathfrak M_x^p(G)$.  Order-unit separation gives a linear
functional $\mathcal L$ with
\[
 \mathcal L(I_p)=1,\qquad
 \mathcal L(\mathfrak M_x^p(G))\subseteq[0,\infty),\qquad
 \mathcal L(F)\leq0.
\]
Apply Lemma~\ref{lem:matrix-gns} and write
$X=(X_1,\ldots,X_n)$ and $T=(T_1,\ldots,T_r)$.  The Hilbert space is
separable, and the real form of the spectral theorem for the abelian von
Neumann algebra generated by the $X_j$ gives a direct-integral decomposition
(equivalently, one may complexify and retain the conjugation-invariant real
form)
\cite[Chapter~14]{KadisonRingrose1997}
\[
 \mathcal H=\int_S^\oplus\mathcal H_x\,d\mu(x),
 \qquad
 X_j=\int_S^\oplus x_jI_{\mathcal H_x}\,d\mu(x),
\]
over a compact set $S\subseteq\R^n$.  Because every $T_k$ commutes with the
spectral projections of $X$, it is decomposable:
\[
 T_k=\int_S^\oplus T_k(x)\,d\mu(x),
\]
where $T_k(x)$ is bounded and self-adjoint for almost every $x$.  The
operator inequalities \eqref{eq:matrix-gns-localization} decompose
fiberwise, so
\begin{equation}\label{eq:fiber-operator-feasibility}
 G_i(x,T(x))\succeq0
 \quad\text{on }\mathbb R^{q_i}\otimes\mathcal H_x
 \quad\text{for almost every }x.
\end{equation}

Fix such an $x$ and a finite-dimensional subspace
$E\subseteq\mathcal H_x$.  Let $P_E$ be the orthogonal projection onto $E$
and set
\[
 A_k\coloneqq P_ET_k(x)|_E\in\Sym^{\dim E}.
\]
Since every pencil is affine in the parameter tuple,
\[
 G_i(x,A)
 =(I_{q_i}\otimes P_E)G_i(x,T(x))
   |_{\mathbb R^{q_i}\otimes E}\succeq0.
\]
Hence the hypothesis \eqref{eq:uniform-free-positive} gives
$F(x,A)\succeq\eps I_{p\dim E}$.  Now take an arbitrary vector
$\eta=(\eta_1,\ldots,\eta_p)\in\mathcal H_x^p$ and choose
$E=\operatorname{span}\{\eta_1,\ldots,\eta_p\}$.  Affineness again gives
\[
 \ip{F(x,T(x))\eta}{\eta}
 =\ip{F(x,A)\eta}{\eta}
 \geq\eps\norm{\eta}^2.
\]
Thus $F(x,T(x))\succeq\eps I$ for almost every $x$, and therefore
$F(X,T)\succeq\eps I$ on $\mathcal H^p$.

With $\xi=(\xi_1,\ldots,\xi_p)$, equations
\eqref{eq:matrix-gns-functional} and
\eqref{eq:matrix-gns-normalization} now yield
\[
 \mathcal L(F)=\ip{F(X,T)\xi}{\xi}
 \geq\eps\norm{\xi}^2=\eps,
\]
contradicting $\mathcal L(F)\leq0$.
\end{proof}

\subsection{A simplex special case}
Ordinary positivity can remain sufficient for several parameters when the
certificate cone records the simplex geometry.  Let the simplex
$\Delta=\conv\{v_0,\ldots,v_r\}$ have barycentric coordinates
$\ell_0,\ldots,\ell_r$, and let
$X=\{x\in\mathbb R^n:a_i(x)\geq0,\ i=1,\ldots,t\}$.

\begin{proposition}\label{prop:simplex-matrix}
Assume that the matrix quadratic module generated by the $a_i$ is Archimedean.  If
$F(x,\la)$ is affine in $\la$ and
$F(x,\la)\succ0$ on $X\times\Delta$, then
\begin{equation}\label{eq:simplex-certificate}
 F(x,\la)=\sum_{k=0}^r\ell_k(\la)
 \left(S_{k0}(x)+\sum_{i=1}^t a_i(x)S_{ki}(x)\right),
\end{equation}
where all $S_{ki}$ are SOS polynomial matrices.
\end{proposition}
\begin{proof}
At every vertex, $F(x,v_k)\succ0$ on $X$.  The Scherer--Hol theorem \cite{SchererHol2006} gives
$F(x,v_k)=S_{k0}+\sum_i a_iS_{ki}$.  Since an affine pencil equals the
barycentric interpolation of its vertex values, summing these identities
against $\ell_k$ proves \eqref{eq:simplex-certificate}.
\end{proof}
The certificate \eqref{eq:simplex-certificate} uses the additional generators
$\ell_k a_i$, which encode the product structure of $X\times\Delta$.  They need not belong to the smaller
quadratic module generated separately by the $a_i$ and $\ell_k$.  If each
vertex matrix $F(x,v_k)$ is itself an SOS matrix, the terms involving $a_i$
are absent and the smaller simplex module suffices.

% The matrix conclusions are summarized in Table~\ref{tab:matrix-summary}.
% \begin{table}[htbp]
% \centering
% \caption{Sufficient positivity conditions for restricted matrix certificates,
% under the respective Archimedean hypotheses. All pointwise and vertexwise
% conditions are strict.}
% \label{tab:matrix-summary}
% \small
% \renewcommand{\arraystretch}{1.12}
% \begin{tabular}{@{}p{.46\linewidth}p{.12\linewidth}p{.34\linewidth}@{}}
% \toprule
% Objective and constraints & Parameters & Positivity condition\\
% \midrule
% Scalar objective, scalar constraints & any $r$ & scalar pointwise positivity\\
% Scalar objective, matrix constraints & any $r$ & scalar pointwise positivity\\
% Matrix objective, scalar or matrix constraints & $r=1$ & matrix pointwise positivity\\
% Matrix objective, scalar or matrix constraints & $r\geq2$ & uniform free positivity\\
% Matrix objective, simplex product cone & any $r$ & vertexwise matrix positivity\\
% \bottomrule
% \end{tabular}
% \end{table}

%% file: sections/07_applications.tex
\section{Applications}\label{sec:applications}
We now apply the preceding theory to optimization problems with affine
parameter dependence. The applications include robust linear programming,
robust Lyapunov inequalities, min-max polynomial optimization, robust polynomial matrix inequalities, and analysis and control of input-affine dynamical systems. Throughout this section, the relevant restricted quadratic modules are assumed to be Archimedean.

% \subsection{Uniform lower bounds and robust feasibility}
% For $p\in\V_r$, a certificate
% \begin{equation}\label{eq:uniform-lower-certificate}
%  p-\gamma\in\M_x(g)
% \end{equation}
% proves that $p(x,\la)\geq\gamma$ for every $(x,\la)\in K$.
% Conversely, if $p>\gamma$ on $K$, Theorem~\ref{thm:main} guarantees
% \eqref{eq:uniform-lower-certificate} at some finite relaxation order.
% Optimizing $\gamma$ therefore gives the convergent moment-SOS hierarchies of
% Section~\ref{sec:hierarchy}, and the SOS multipliers remain independent of the parameter variables.
% In particular, a certificate with $\gamma>0$ gives an explicit uniform positivity margin.

% The same observation applies to a family
% \[
%  h_c(x,\la)=h_0(x,\la)+\sum_{\ell=1}^N c_\ell h_\ell(x,\la),
% \]
% where the coefficients depend affinely on a decision vector
% $c\in\R^N$.  At a fixed relaxation order, the identity
% \[
% h_c=\sigma_0+\sum_i\sigma_i g_i
% \]
% is affine in $c$ and in the Gram matrices of the $\sigma_i$'s.  Thus a linear
% objective and affine side constraints on $c$ produce an SDP. One may also
% maximize a common margin $\gamma$ subject to
% $h_c-\gamma\in\M_{x,d}(g)$, thereby jointly selecting a design and certifying
% its worst-case performance.

\subsection{Robust linear programming}
Consider the robust linear programming problem\footnote{Here, $\la$ denotes the decision vector of the linear program, while $x$ denotes the parameter vector.}:
\begin{equation}\label{eq:robust-lp}
f^*\coloneqq\max_{x\in X}\min_{\lambda} \{f(x,\lambda):g_i(x,\la)\geq0,\ i\in[m]\},
\end{equation}
where $f,g_i\in\V_r$ are affine in $\lambda$ and $X\subset\mathbb R^n$ is a compact basic semialgebraic set. Problem~\eqref{eq:robust-lp} is a special case of the robust polynomial optimization problem studied in \cite{lasserre2011min}. Following \cite{lasserre2011min}, we develop a semidefinite relaxation approach to \eqref{eq:robust-lp} using the restricted-multiplier Positivstellens\"atze.

Let $\Omega\subset\mathbb R^n$ be a simple set (e.g., a box or an ellipsoid) that contains $X$ and is defined by $g_i(x)\ge0$, $i=m+1,\dots,m+p$. Set $g\coloneqq(g_i)_{i=1}^{m+p}$. We make the following assumption:
\begin{assumption}\label{assump}
For every $x\in\Omega$, the set $\Lambda(x)= \{\la\in\R^r:g_i(x,\la)\geq0,\ i\in[m]\}$ is nonempty.
\end{assumption}
The optimal value function is defined by
\begin{equation}\label{eq:valuefunction}
 \psi(x)\coloneqq \min_{\la\in\Lambda(x)}f(x,\la),\qquad x\in \Omega,
\end{equation}
which is lower semicontinuous. Every polynomial $p\in\R[x]$ satisfying $f-p\in\M_x(g)$ obeys $p\leq\psi$ on $\Omega$. 
% This gives a convergent method for approximating $\psi$ from below in an integral sense.
The following theorem is a specialization of Theorems 2.2 and 2.4 of \cite{lasserre2011min}.
\begin{theorem}
\label{thm:robust-lp}
Suppose that $\M_x(g)$ is restricted Archimedean and Assumption~\ref{assump} holds.
Let $\nu$ be the Borel probability measure with uniform distribution on $\Omega$, and let $\phi_d(x)$ be an optimal or nearly optimal solution (e.g., such that $\int_\Omega \phi_d\,d\nu\ge\tau_d-1/d$) of
\begin{equation}\label{eq:value-approx-sdp}
\tau_d\coloneqq\sup\left\{\int_\Omega p\,d\nu:
 p\in\R[x]_{\leq2d},\ f-p\in\M_{x,d}(g)\right\}.
\end{equation}
Then as $d\to\infty$,
\[
\int_{\Omega}|\psi(x)-\phi_d(x)|\,d\nu(x)\to 0.
\]
Moreover, letting $v_d\coloneqq \max_{x\in X}\phi_d(x)$ and $\hat{v}_d\coloneqq \max_{i\le d}\{v_i\}$, one has $\hat{v}_d\to f^*$ as $d\to\infty$.
\end{theorem}

% \subsection{Region inclusion and universal quantifiers}
% Let
% \[
%  D\coloneqq \{x\in\R^n:a_j(x)\geq0,\ j=1,\ldots,s\}
% \]
% be a prescribed compact region, and consider
% \[
%  \mathcal R\coloneqq \{x:h(x,\la)\geq0
%           \text{ for every }\la\in\Lambda(x)\}.
% \]
% If
% \begin{equation}\label{eq:region-certificate}
%  h\in\M_x(a,g),
% \end{equation}
% then $D\subseteq\mathcal R$.  Conversely, if $h$ has a strictly positive
% margin on
% \[
%  \{(x,\la):x\in D,\ \la\in\Lambda(x)\}
% \]
% and the enlarged module is restricted Archimedean, then
% \eqref{eq:region-certificate} holds at some finite degree.

% When $h=h_c$ depends affinely on design variables, maximizing a linear
% functional of $c$ subject to a fixed-degree version of
% \eqref{eq:region-certificate} remains an SDP.  This yields inner
% approximations of robustly feasible design sets and certified regions of
% validity for parameter-dependent inequalities.

\subsection{Robust Lyapunov and dissipation inequalities}
Consider a polynomial vector field affine in the uncertainty:
\[
 \dot x=F(x,\la)
       =F_0(x)+\sum_{k=1}^r\la_kF_k(x).
\]
For a polynomial Lyapunov candidate $V$, its Lie derivative
$\nabla V(x)^{\intercal}F(x,\la)$ is affine in $\la$.  On a prescribed compact region,
a restricted SOS certificate
\begin{equation}\label{eq:lyap}
 -\nabla V(x)^{\intercal}F(x,\la)-w(x)\in\M_x(g)
\end{equation}
proves the uniform decay estimate $\dot V\leq-w$ for every admissible parameter.

This condition also leads to a convex synthesis problem. Choose a fixed positive-definite polynomial $q$, and
set $w=\eta q$ with an unknown decay margin $\eta\geq0$. Positivity of $V$
on the state region can be imposed by
\[
 V(x)-\epsilon q(x)\in\M_x(a),
 \qquad \epsilon>0,
\]
where the $a_i$'s describe the state region. If the vector field and the
uncertainty generators are fixed, coefficient matching in this constraint
and in \eqref{eq:lyap} is affine in the coefficients of $V$, the margin
$\eta$, and the SOS Gram matrices. Maximizing $\eta$ therefore gives an SDP
for constructing a common polynomial Lyapunov function with a certified
worst-case decay rate.

\subsection{Min-max polynomial optimization}
We consider a design variable $z\in Z\subseteq\R^p$ that is
distinct from the variables $(x,\la)$ in the uncertainty set $K$. Assume
that $Z$ and $K$ are nonempty and compact,
and that $\M_x(g)$ is restricted Archimedean.  Let
\[
 f(z,x,\la)=f_0(z,x)+\sum_{k=1}^r\la_k f_k(z,x),
 \qquad f_k\in\R[z,x].
\]
The min-max optimization problem is
\begin{equation}\label{eq:app-minmax}
 f^*\coloneqq \min_{z\in Z}\Psi(z),
 \qquad
 \Psi(z)\coloneqq \max_{(x,\la)\in K}f(z,x,\la).
\end{equation}
The function $\Psi$ is continuous, since $f$ is continuous and the compact
set $K$ is fixed.  In particular, the outer minimum is attained.

For a fixed design $z$, the epigraph constraint $\Psi(z)\leq\gamma$
requires $\gamma-f(z,x,\la)\geq0$ for every $(x,\la)\in K$.
A sufficient condition is the restricted certificate
\begin{equation}\label{eq:app-minmax-certificate}
 \gamma-f(z,x,\la)
 =\sigma_0(x)+\sum_{i=1}^m\sigma_i(x)g_i(x,\la),
 \qquad \sigma_i\in\Sigma[x].
\end{equation}
Here $z$ is a decision vector: the equality is a polynomial identity in
$(x,\la)$ after choosing $z$.
Conversely, whenever $\gamma>\Psi(z)$,
Theorem~\ref{thm:main} guarantees
\eqref{eq:app-minmax-certificate} at some finite degree.
Using the truncated modules of Section~\ref{sec:hierarchy}, define
\begin{equation}\label{eq:app-minmax-hierarchy}
 u_d\coloneqq \inf\{\gamma:z\in Z,
       \gamma-f(z,x,\la)\in\M_{x,d}(g)\}.
\end{equation}
Every feasible pair $(z,\gamma)$ supplies both an admissible design and
the certified worst-case bound $\Psi(z)\leq\gamma$. When $f$ is affine in $z$ and $Z$ is spectrahedral (or has a semidefinite representation using auxiliary
variables), \eqref{eq:app-minmax-hierarchy} is an SDP.

% The corresponding max--min problem and its restricted hierarchy are
% \begin{align}
%  V_{\max\min}
%  &\coloneqq \max_{z\in Z}\Phi(z),\qquad
%    \Phi(z)\coloneqq \min_{(x,\la)\in K}f(z,x,\la),
%    \label{eq:app-maxmin}\\
%  L_d&\coloneqq \sup\{\gamma:z\in Z,
%        f(z,x,\la)-\gamma\in\M_{x,d}(g)\}.
%    \label{eq:app-maxmin-hierarchy}
% \end{align}
% The function $\Phi$ is also continuous.  A feasible pair in
% \eqref{eq:app-maxmin-hierarchy} guarantees $\Phi(z)\geq\gamma$.

\begin{proposition}
\label{prop:app-minmax-convergence}
Under the preceding assumptions, $u_d\searrow f^*$ as $d\to\infty$.
\end{proposition}
\begin{proof}
Nonnegativity of the generators and multipliers on $K$ gives
$f^*\leq u_d$.
Let $z^\star$ attain the minimum in \eqref{eq:app-minmax}.  For every
$\eps>0$,
\[
 f^*+\eps-f(z^\star,x,\la)\geq\eps
 \qquad ((x,\la)\in K).
\]
Theorem~\ref{thm:main} gives a restricted certificate at a finite order,
and hence $u_d\leq f^*+\eps$ for all sufficiently large $d$.
Letting $\eps\to0$ gives the desired result.
\end{proof}

\subsection{Robust polynomial matrix inequalities through scalarization}
Let $F(x,\la)\in\Sym^p[x,\la]$ be affine in $\la$, and define
\[
 \widetilde K\coloneqq 
 \{(x,u,\la):(x,\la)\in K,\ \|u\|^2=1\}.
\]
Strict matrix positivity of $F$ on $K$ is equivalent to
$u^{\intercal}F(x,\la)u>0$ on $\widetilde K$.
Treating $(x,u)$ as the principal variables, the original ball certificate
and the sphere equality give an order bound in $(x,u)$, while the parameter
bounds are unchanged.  Theorem~\ref{thm:equality} therefore gives
\begin{equation}\label{eq:scalarization}
 u^{\intercal}F(x,\la)u
 =\sigma_0(x,u)+\sum_i\sigma_i(x,u)g_i(x,\la)
   +q(x,u)(1-\|u\|^2),
\end{equation}
where the $\sigma_i$'s are SOS and independent of $\la$.
This provides a scalar restricted certificate for a robust polynomial
matrix inequality with any finite number of affine parameters.
The multipliers in \eqref{eq:scalarization} may have high degree in $u$, so
the identity need not collapse to a matrix certificate with SOS-matrix
multipliers depending only on $x$.

\subsection{Analysis and control of input-affine dynamical systems}
\label{subsec:input-affine-systems}
The restricted Positivstellensatz also applies to auxiliary-function
formulations for dynamical systems.  Miller and Sznaier \cite{MillerSznaier2025} eliminated affine uncertainty from Lie inequalities by
parameterized robust counterparts, with continuous and, under additional structural assumptions, polynomial conic multiplier selections.  Here we provide a related decomposition directly from restricted quadratic-module membership, including for
polynomially varying input constraints.

Let $T>0$ and consider
\begin{equation}\label{eq:ia-dynamics}
 \dot x(t)=F(t,x(t),w(t)),\qquad
 F(t,x,w)=F_0(t,x)+\sum_{k=1}^r w_kF_k(t,x),
\end{equation}
where $F_k\in\R[t,x]^n$.  Write $y=(t,x)$ and let
\[
 Y=[0,T]\times X=\{y:a_j(y)\geq0,\ j=1,\ldots,s\},\qquad
 W(y)=\{w:g_i(y,w)\geq0,\ i=1,\ldots,m\},
\]
with $g_i(y,w)=g_{i0}(y)+\sum_{k=1}^r w_kg_{ik}(y)$.
Assume that $Y$ is compact, every $W(y)$ is nonempty, and the restricted quadratic module
\[
 \M_y(a,g)\coloneqq\left\{\sigma_0+\sum_{j=1}^s\eta_ja_j
                +\sum_{i=1}^m\sigma_ig_i:
       \sigma_0,\eta_j,\sigma_i\in\Sigma[y]\right\}
\]
is restricted Archimedean.  In particular, the joint feasible set
$K_Y=\{(y,w):y\in Y,\ w\in W(y)\}$ is compact.
We consider absolutely continuous trajectories that satisfy
\eqref{eq:ia-dynamics} almost everywhere, are driven by measurable admissible inputs,
and remain in $X$ on the time interval under consideration.
Thus the conclusions below concern state-constrained trajectories; they
apply to all trajectories from a given initial set if these remain in $X$.

For $v\in\R[t,x]$, define the Lie derivative
\begin{equation}\label{eq:ia-lie}
 \mathscr L_Fv(y,w)
 =\partial_tv(y)+\nabla_xv(y)^{\intercal}F_0(y)
       +\sum_{k=1}^r w_k\nabla_xv(y)^{\intercal}F_k(y).
\end{equation}
This polynomial is affine in $w$.
Thus Theorem~\ref{thm:main} implies that
$-\mathscr L_Fv>0$ on $K_Y$ if and only if
there exist $\epsilon>0$ and SOS polynomials $\sigma_i,\eta_j$ in $y$ such that
\begin{equation}\label{eq:ia-lie-certificate}
 -\mathscr L_Fv(y,w)-\epsilon
 =\sigma_0(y)+\sum_{j=1}^s\eta_j(y)a_j(y)
                  +\sum_{i=1}^m\sigma_i(y)g_i(y,w).
\end{equation}
Every such identity proves
$v(t,x(t))\leq v(0,x(0))-\epsilon t$ along admissible trajectories.
Matching coefficients of $w$ gives
\begin{align}
 -\partial_tv-\nabla_xv^{\intercal}F_0-\epsilon
   &=\sigma_0+\sum_j\eta_ja_j+\sum_i\sigma_ig_{i0},
       \label{eq:ia-match-constant}\\
 -\nabla_xv^{\intercal}F_k&=\sum_i\sigma_ig_{ik},\qquad k=1,\ldots,r.
       \label{eq:ia-match-input}
\end{align}
For fixed dynamics and generators, these equations are affine in the
coefficients of $v$, the margin, and the SOS Gram matrices.  Any fixed-degree
formulation with a linear objective and affine side constraints is therefore
an SDP.

\paragraph{Polyhedral inputs and robust counterparts}
For a fixed nonempty compact polytope $W=\{w:Aw\leq b\}$,
\eqref{eq:ia-match-constant}--\eqref{eq:ia-match-input} become
\begin{equation}\label{eq:ia-polytope-counterpart}
 A^{\intercal}\sigma=
 \begin{pmatrix}\nabla_xv^{\intercal}F_1\\ \vdots\\ \nabla_xv^{\intercal}F_r\end{pmatrix},
 \qquad
 -\partial_tv-\nabla_xv^{\intercal}F_0-b^{\intercal}\sigma-\epsilon\in\M_y(a),
 \qquad \sigma_i\in\Sigma[y].
\end{equation}
The equations in \eqref{eq:ia-polytope-counterpart} have the Farkas-multiplier
form of the polyhedral robust counterpart in
\cite[Section~4.1]{MillerSznaier2025}. Note that \cite[Theorem~3.2]{MillerSznaier2025} assumes that the matrices
defining the conic input representation are constant in $y$, whereas
equations~\eqref{eq:ia-match-constant}--\eqref{eq:ia-match-input} allow
polynomial $g_{ik}(y)$ under restricted Archimedeanity, without constructing
a continuous selection of pointwise dual multipliers.

\paragraph{Spectrahedral inputs}
Suppose that the input constraints are instead
$G_i(y,w)=G_{i0}(y)+\sum_{k=1}^r w_kG_{ik}(y)\succeq0$.
Under the corresponding matrix restricted-Archimedean hypothesis, the
scalar-objective, matrix-constraint version of the restricted Positivstellensatz gives
\begin{equation}\label{eq:ia-matrix-input}
 -\mathscr L_Fv-\epsilon
 =\sigma_0+\sum_j\eta_ja_j+\sum_i\langle S_i,G_i\rangle,
 \qquad \sigma_0,\eta_j\in\Sigma[y],
\end{equation}
where each $S_i(y)$ is an SOS matrix.
Here $\epsilon>0$ is available when $-\mathscr L_Fv$ is strictly positive.
Coefficient matching gives
$-\nabla_xv^{\intercal}F_k=\sum_i\langle S_i,G_{ik}\rangle$.

For each basic semialgebraic set $D$ used below, $Q(D)$ denotes the
ordinary quadratic module associated with $D$.

\paragraph{Peak, distance, reachability, and control}
Let $X_0\subseteq X$ be a nonempty compact basic semialgebraic initial set,
and let $p\in\R[x]$. Imposing
\begin{equation}\label{eq:ia-peak-inequalities}
 \gamma-v(0,x)\in Q(X_0),\qquad
 v(t,x)-p(x)\in Q(Y),\qquad
 -\mathscr L_Fv\in\M_y(a,g)
\end{equation}
gives $p(x(t))\leq v(t,x(t))\leq v(0,x(0))\leq\gamma$. Thus minimizing $\gamma$ gives certified upper bounds on the peak.

For a compact basic semialgebraic unsafe set $U$ and a polynomial cost
$c(x,z)$, impose
$\gamma\geq v(0,x)$ on $X_0$ and
$v(t,x)+c(x,z)\geq0$ on $Y\times U$.
Together with $-\mathscr L_Fv\in\M_y(a,g)$, these imply
$c(x(t),z)\geq-\gamma$ for every admissible trajectory and $z\in U$.
Thus minimizing $\gamma$ gives certified lower bounds on the minimum cost
of closest approach, including squared Euclidean distance.

For forward reachability, the conditions
\[
 -v(0,\cdot)\in Q(X_0),\qquad
 -\mathscr L_Fv\in\M_y(a,g)
\]
give $\mathcal R_t\subseteq\{x\in X:v(t,x)\leq0\}$, where
$\mathcal R_t$ is the set of states reachable at time $t$.
If additionally $\phi\in Q(X)$ and
$\phi+v(T,\cdot)-1\in Q(X)$, then
$\phi\geq1$ on $\mathcal R_T$, and minimizing
$\int_X\phi(x)\,dx$ yields certified upper bounds on its volume.
For a compact basic semialgebraic target $X_T$, instead impose
\[
 v(T,\cdot)\in Q(X_T),\qquad
 -\mathscr L_Fv\in\M_y(a,g),\qquad
 \phi\in Q(X),\quad \phi-1-v(0,\cdot)\in Q(X).
\]
Every state that can be steered to $X_T$ at time $T$ then satisfies
$v(0,x)\geq0$ and $\phi(x)\geq1$.  This gives the corresponding outer
approximation of the finite-horizon backward reachable set.

For optimal control with running cost $\ell(y,w)$ affine in $w$ and
terminal cost $h$, one instead imposes
\begin{equation}\label{eq:ia-running-cost}
 \mathscr L_Fv+\ell\in\M_y(a,g),\qquad
 h-v(T,\cdot)\in Q(X_T).
\end{equation}
Integration gives
$v(0,x_0)\leq\int_0^{T}\ell(t,x(t),w(t))\,dt+h(x(T))$
for every admissible control ending in $X_T$.

% \paragraph{Discrete time and data-driven models}
% For $x^+=F_0(x)+\sum_kw_kF_k(x)$, the expression
% $v(F(x,w))$ is generally nonlinear in $w$.  Introduce a next-state
% variable $z$ in a prescribed compact set containing the admissible images,
% and the affine equalities
% \[
%  h_\ell(x,z,w)=z_\ell-F_{0\ell}(x)
%                    -\sum_kw_kF_{k\ell}(x)=0,\qquad \ell=1,\ldots,n.
% \]
% With $(x,z)$ as principal variables, Theorem~\ref{thm:equality} certifies
% strict positivity of $\alpha+v(x)-v(z)$ on the transition set by
% \begin{equation}\label{eq:ia-discrete-certificate}
%  \begin{aligned}
%  \alpha+v(x)-v(z)
%   ={}&\sigma_0(x,z)+\sum_j\eta_j(x,z)a_j(x,z)\\
%     &+\sum_i\sigma_i(x,z)g_i(x,w)
%        +\sum_{\ell=1}^n r_\ell(x,z)h_\ell(x,z,w),
%  \end{aligned}
% \end{equation}
% provided the equality-augmented module is restricted Archimedean.
% The multipliers $\sigma_0,\eta_j,\sigma_i$ are SOS, while
% $r_\ell\in\R[x,z]$ are unrestricted.  Here the $a_j$ describe the
% principal-variable domain.  The certificate implies
% $v(x_N)\leq v(x_0)+N\alpha$; with $v\geq p$, $\gamma\geq v$ on $X_0$,
% and $\alpha\geq0$, it yields the peak bound $\gamma+N\alpha$ through
% $N$ steps.  This is the next-state lifting used in
% \cite[Section~5]{MillerSznaier2025}, now combined with the equality version
% of the restricted Positivstellensatz.

%% file: sections/08_experiments.tex
\section{Numerical experiments}\label{sec:experiments}
In this section, we compare the full joint-variable and restricted SOS formulations in terms of relaxation strength, PSD block size, and computational cost. The tests cover
polynomial optimization under box, simplex, and $x$-dependent fiber uncertainty, robust Lyapunov design, and min-max polynomial optimization.
All optimization models were implemented in Julia using {\tt TSSOS 1.5.3}\footnote{{\tt TSSOS} is freely available at \href{https://github.com/wangjie212/TSSOS}{https://github.com/wangjie212/TSSOS}.}~\cite{magron2021tssos} and solved with {\tt MOSEK 11.0}. The experiments were performed on a desktop computer with an Intel Core i9-10900 CPU at 2.80~GHz and 64~GB of RAM.
Throughout this section, ``$N_{\rm full}$'' and ``$N_{\rm res}$'' denote the largest PSD block sizes in the full and restricted formulations, respectively; ``$t_{\rm full}$'' and ``$t_{\rm res}$'' denote the solution times in seconds for the full and restricted formulations, respectively.

\subsection{Box uncertainty}
We begin with the family
\begin{equation}\label{eq:box-benchmark}
\begin{aligned}
 f_{\min}&=\min\left\{f(x,\lambda):
 x\in[-1,1]^n,\ \lambda\in[-1,1]^r\right\},\\
 f(x,\lambda)&=p_0(x)+\sum_{k=1}^r\lambda_kp_k(x),
\end{aligned}
\end{equation}
where $p_0,\ldots,p_r$ are dense cubic polynomials.
For each $(n,r)\in\{4,6,8\}\times\{2,4,6\}$, we generated five instances in which the coefficients of each polynomial were sampled uniformly from $[-1,1]$, after which each polynomial was normalized by its largest coefficient magnitude. The box is described by $1-x_i^2\geq0$, $i=1,\ldots,n$, and $1\pm\lambda_k\geq0$,
$k=1,\ldots,r$.

We solve the full and restricted SOS relaxations for \eqref{eq:box-benchmark} at order $d=3$.
Table~\ref{tab:box-performance} reports median recorded times over instances
for which both formulations returned a bound, together with the ratio of
these medians. The restricted formulation succeeded on all $45$ instances, while the full formulation returned
bounds on $42$ instances and ran out of memory on the remaining three with $(n,r)=(8,6)$. The full and restricted bounds agree to within $10^{-5}$ on $38$ of the $42$ jointly solved instances. On the
remaining four instances, the full relaxation improves the lower bound by between $1.11\times10^{-2}$ and $1.09\times10^{-1}$.
The median speedup ranges from $5.9\times$ to $532.5\times$ across the tested configurations. For each fixed $n$, the computational advantage generally grows with the number of uncertainty parameters, reflecting the rapid growth of the joint-variable PSD blocks avoided by the restricted formulation.

\begin{table}[t]
\centering
\caption{Comparison under box uncertainty, $\theta\coloneqq f^{\rm full}_3-f^{\rm res}_3$.}
\label{tab:box-performance}
\begin{tabular}{ccrrrrrr}
\toprule
$n$ & $r$ & $N_{\rm full}$ & $N_{\rm res}$ & $t_{\rm full}$ & $t_{\rm res}$
& speedup & $\max|\theta|$ \\
\midrule
4 & 2 &  84 &  35 &   0.147 & 0.025 &   $5.9\times$ & $1.58\times10^{-2}$ \\
4 & 4 & 165 &  35 &   1.676 & 0.038 &  $44.1\times$ & $1.09\times10^{-1}$ \\
4 & 6 & 286 &  35 &  14.013 & 0.046 & $304.6\times$ & $8.53\times10^{-2}$ \\
6 & 2 & 165 &  84 &   1.640 & 0.146 &  $11.2\times$ & $8.51\times10^{-7}$ \\
6 & 4 & 286 &  84 &  12.272 & 0.215 &  $57.1\times$ & $1.11\times10^{-2}$ \\
6 & 6 & 455 &  84 & 109.164 & 0.205 & $532.5\times$ & $1.11\times10^{-6}$ \\
8 & 2 & 286 & 165 &  11.435 & 1.732 &   $6.6\times$ & $4.73\times10^{-7}$ \\
8 & 4 & 455 & 165 &  94.024 & 2.066 &  $45.5\times$ & $9.80\times10^{-7}$ \\
8 & 6 & 680 & 165 & 911.939 & 2.387 & $382.1\times$ & $9.79\times10^{-9}$ \\
\bottomrule
\end{tabular}
\end{table}

We next examine how the bounds for \eqref{eq:box-benchmark} change with the relaxation order.
In this experiment, $p_0,\ldots,p_r$ are chosen to be dense quartic polynomials.
We generated three instances for each of $(n,r)=(3,2)$ and $(4,3)$, and reused each fixed instance at orders $d=3,4,5$, giving $18$ paired relaxation comparisons. The results are presented in Table~\ref{tab:convergence-performance}.
For these instances, the restricted formulation
occasionally requires one additional relaxation order to reach the
same bound as the full formulation, while substantially reducing PSD block sizes. For example, for $(n,r)=(4,3)$, increasing $d$ from
three to five enlarges the main full block from $120$ to $792$, compared
with $35$ to $126$ for the restricted formulation.  The corresponding
median recorded times are $327.195$ and $0.730$ seconds, respectively.

\begin{table}[t]
\centering
\caption{Comparison of convergence, $\theta\coloneqq f_d^{\rm full}-f_d^{\rm res}$.}
\label{tab:convergence-performance}
\setlength{\tabcolsep}{4pt}
\begin{tabular}{cccrrrrrr}
\toprule
$n$ & $r$ & $d$ & $N_{\rm full}$ & $N_{\rm res}$
& $t_{\rm full}$ & $t_{\rm res}$ & speedup & $\max|\theta|$ \\
\midrule
3 & 2 & 3 &  56 &  20 &   0.126 & 0.016 &   $7.9\times$ & $3.44\times10^{-1}$ \\
3 & 2 & 4 & 126 &  35 &   0.730 & 0.063 &  $11.6\times$ & $6.07\times10^{-7}$ \\
3 & 2 & 5 & 252 &  56 &   4.946 & 0.111 &  $44.6\times$ & $1.30\times10^{-6}$ \\
4 & 3 & 3 & 120 &  35 &   0.826 & 0.032 &  $25.8\times$ & $4.66\times10^{-8}$ \\
4 & 3 & 4 & 330 &  70 &  14.427 & 0.190 &  $75.9\times$ & $1.05\times10^{-6}$ \\
4 & 3 & 5 & 792 & 126 & 327.195 & 0.730 & $448.2\times$ & $1.38\times10^{-6}$ \\
\bottomrule
\end{tabular}
\end{table}

\subsection{Simplex uncertainty}
We now replace the box by a simplex and consider
\begin{equation}\label{eq:simplex-benchmark}
\begin{aligned}
 f_{\min}&=\min\left\{f(x,\lambda):
 x\in[-1,1]^n,\ \lambda\in\Delta_r\right\},\\
 f(x,\lambda)&=\left(1-\sum_{k=1}^r\lambda_k\right)p_0(x)
 +\sum_{k=1}^r\lambda_kp_k(x),\\
 \Delta_r&=\left\{\lambda\in\mathbb R^r:
 \lambda_k\geq0,\ k=1,\ldots,r,\ 1-\sum_{k=1}^r\lambda_k\geq0\right\},
\end{aligned}
\end{equation}
where $p_0,\ldots,p_r$ are dense quartic polynomials. 
For each $(n,r)\in\{4,6,8\}\times\{2,4,6\}$, we generated five instances in which the coefficients of each polynomial were sampled uniformly from $[-1,1]$, after which each polynomial was normalized by its largest coefficient magnitude.

We solve the full and restricted SOS relaxations for \eqref{eq:simplex-benchmark} at order $d=3$.
Table~\ref{tab:simplex-performance} reports median recorded times over instances
for which both formulations returned a bound, together with the ratio of these medians for each configuration. The restricted formulation succeeded on all $45$ instances, while the full formulation returned
bounds on $41$ instances and ran out of memory on the remaining four with $(n,r)=(8,6)$. The two bounds agree to within $10^{-5}$ on $38$ of the $41$ jointly solved instances. On the
remaining three instances, the full relaxation improves the lower bound by between $8.04\times10^{-3}$ and $8.28\times10^{-3}$. The recorded speedup reaches $208.4\times$ for $(n,r)=(6,6)$.
For the single recorded $(8,6)$ instance, the full and restricted times are
$927.056$ and $4.515$ seconds, respectively, while the largest PSD block
size decreases from $680$ to $165$. As in the box experiments, the
restricted formulation substantially reduces block sizes, with a
possible loss in bound strength at a fixed order.

\begin{table}[t]
\centering
\caption{Comparison under simplex uncertainty, $\theta\coloneqq f^{\rm full}_3-f^{\rm res}_3$.}
\label{tab:simplex-performance}
\begin{tabular}{ccrrrrrr}
\toprule
$n$ & $r$ & $N_{\rm full}$ & $N_{\rm res}$ & $t_{\rm full}$ & $t_{\rm res}$
& speedup & $\max|\theta|$ \\
\midrule
4 & 2 &  84 &  35 &   0.138 & 0.023 &   $6.0\times$ & $2.23\times10^{-7}$ \\
4 & 4 & 165 &  35 &   1.293 & 0.028 &  $46.2\times$ & $3.28\times10^{-8}$ \\
4 & 6 & 286 &  35 &   8.953 & 0.050 & $179.1\times$ & $8.28\times10^{-3}$ \\
6 & 2 & 165 &  84 &   1.518 & 0.199 &   $7.6\times$ & $8.18\times10^{-3}$ \\
6 & 4 & 286 &  84 &  10.321 & 0.230 &  $44.9\times$ & $5.54\times10^{-8}$ \\
6 & 6 & 455 &  84 &  77.507 & 0.372 & $208.4\times$ & $4.44\times10^{-7}$ \\
8 & 2 & 286 & 165 &  11.098 & 1.980 &   $5.6\times$ & $4.69\times10^{-7}$ \\
8 & 4 & 455 & 165 &  85.914 & 2.330 &  $36.9\times$ & $1.28\times10^{-7}$ \\
8 & 6 & 680 & 165 & 927.056 & 4.515 & $205.3\times$ & $2.01\times10^{-9}$ \\
\bottomrule
\end{tabular}
\end{table}

% Simplex uncertainty also provides an independent reference calculation.
% Since $f$ is affine in $\lambda$,
% \begin{equation}\label{eq:simplex-vertex-reference}
%  \min_{\lambda\in\Delta_r}f(x,\lambda)=\min_{0\leq k\leq r}p_k(x),
%  \qquad
%  f_{\min}=\min_{0\leq k\leq r}\min_{x\in[-1,1]^n}p_k(x).
% \end{equation}
% The minimum of the SOS lower bounds for these $r+1$ polynomial optimization
% problems in $x$ therefore gives a reference lower bound.  The recorded vertex
% reference computations all succeeded.  Their bounds differ from the standard
% bounds by at most $1.39\times10^{-6}$ and agree with the restricted bounds to
% within $10^{-5}$ on the same $38$ instances.  This comparison corroborates
% the observed agreement of the relaxations, but does not by itself certify
% attainment of the global optimum.

\subsection{\texorpdfstring{$x$}{x}-dependent uncertainty fibers}
To test the framework beyond Cartesian-product sets, consider
\begin{equation}\label{eq:xdependent-benchmark}
\begin{aligned}
 f_{\min}&=\min\left\{f(x,\lambda):
 x\in[-1,1]^n,\ \lambda\in\Lambda(x)\right\},\\
 f(x,\lambda)&=p_0(x)+\sum_{k=1}^r\lambda_kp_k(x),\\
 \Lambda(x)&=\left\{\lambda\in\mathbb R^r:
 -h_k(x)\leq\lambda_k\leq h_k(x),\ k=1,\ldots,r\right\},\\
 h_k(x)&=1-\alpha x_{k'}^2,\qquad
 k'=1+\operatorname{mod}(k-1,n).
\end{aligned}
\end{equation}
Here $p_0,\ldots,p_r$ are dense cubic polynomials and
$\alpha=0.5$.  Since $1/2\leq h_k(x)\leq1$ on $[-1,1]^n$, every fiber is
nonempty and uniformly bounded.  For each
$(n,r)\in\{4,6,8\}\times\{2,4,6\}$, we generated five instances in which the coefficients of each polynomial were sampled uniformly from $[-1,1]$, after which each polynomial was normalized by its largest coefficient magnitude.  The defining
inequalities are $1-x_i^2\geq0$ and $h_k(x)\pm\lambda_k\geq0$.

We solve the full and restricted SOS relaxations for \eqref{eq:xdependent-benchmark} at order $d=3$.
Table~\ref{tab:xdependent-performance} reports median recorded times over instances
for which both formulations returned a bound, together with the ratio of these medians for each configuration.
The restricted formulation succeeded on all $45$ instances, while the full formulation returned a bound on $43$ instances and reported
an out-of-memory error on two of the five instances with $(n,r)=(8,6)$.
The two bounds agree to within $10^{-5}$ on $34$ of the $43$ jointly solved instances.  On the remaining nine instances, the full formulation
improves the lower bound by between $1.97\times10^{-3}$ and
$2.48\times10^{-1}$.
The reduction remains substantial for varying fibers. At $(n,r)=(6,6)$,
the median recorded times are $113.476$ and $0.420$ seconds, a ratio of
$270.2\times$. At $(8,6)$, the largest block size decreases from $680$
to $165$, with paired median times of $922.642$ and $4.509$ seconds.

\begin{table}[t]
\centering
\caption{Comparison under $x$-dependent uncertainty fibers, $\theta\coloneqq f^{\rm full}_3-f^{\rm res}_3$.}
\label{tab:xdependent-performance}
\setlength{\tabcolsep}{4pt}
\begin{tabular}{ccrrrrrr}
\toprule
$n$ & $r$ & $N_{\rm full}$ & $N_{\rm res}$ & $t_{\rm full}$ & $t_{\rm res}$
& speedup & $\max|\theta|$ \\
\midrule
4 & 2 & 84 & 35 & 0.251 & 0.070 & $3.6\times$ & $1.58\times10^{-7}$ \\
4 & 4 & 165 & 35 & 2.583 & 0.130 & $19.9\times$ & $1.47\times10^{-1}$ \\
4 & 6 & 286 & 35 & 13.817 & 0.119 & $116.1\times$ & $1.10\times10^{-1}$ \\
6 & 2 & 165 & 84 & 2.551 & 0.272 & $9.4\times$ & $1.75\times10^{-7}$ \\
6 & 4 & 286 & 84 & 15.092 & 0.364 & $41.5\times$ & $2.48\times10^{-1}$ \\
6 & 6 & 455 & 84 & 113.476 & 0.420 & $270.2\times$ & $5.52\times10^{-7}$ \\
8 & 2 & 286 & 165 & 13.208 & 3.066 & $4.3\times$ & $1.56\times10^{-2}$ \\
8 & 4 & 455 & 165 & 115.282 & 4.523 & $25.5\times$ & $2.08\times10^{-1}$ \\
8 & 6 & 680 & 165 & 922.642 & 4.509 & $204.6\times$ & $6.52\times10^{-7}$ \\
\bottomrule
\end{tabular}
\end{table}

\subsection{Robust Lyapunov design}
Next, we seek a common polynomial Lyapunov function for the uncertain
system
\begin{equation}\label{eq:control-benchmark}
 \dot x=F(x,\lambda)
 =\left(1-\sum_{k=1}^r\lambda_k\right)F_0(x)
 +\sum_{k=1}^r\lambda_kF_k(x),
 \qquad \lambda\in\Delta_r,
\end{equation}
where $\Delta_r$ is the simplex in~\eqref{eq:simplex-benchmark}.
The generated cubic vector fields have the form
\[
 F_j(x)=-D_jx+K_jx+q_j(x)+\beta_j\|x\|^2x,
 \qquad j=0,\ldots,r,
\]
where $D_j$ is a positive diagonal matrix, $K_j$ is skew-symmetric, and
$q_j$ is a homogeneous cubic field satisfying $x^\intercal q_j(x)=0$.
The positive coefficients $\beta_j$ satisfy
$\beta_jR^2<\lambda_{\min}(D_j)$.  Consequently,
$V_0(x)=\|x\|^2$ satisfies
\begin{equation}\label{eq:control-quadratic-margin}
 -\nabla V_0(x)^\intercal F(x,\lambda)
 \geq \delta\|x\|^2,
 \qquad
 \delta
 \coloneqq2\min_{0\leq j\leq r}
 \bigl\{\lambda_{\min}(D_j)-\beta_jR^2\bigr\}>0
\end{equation}
on $\{x:\|x\|\leq R\}\times\Delta_r$.  This gives an explicit feasible
reference margin for every generated instance.

The experiments use $R=1$, five random seeds for each
$(n,r)\in\{(4,2),(4,3),(6,2),(6,3)\}$, and common Lyapunov
polynomials of degree at most four.  To remove the scaling ambiguity, the
quadratic part is fixed:
\[
 V(x)=\|x\|^2+\sum_{3\leq|\alpha|\leq4}c_\alpha x^\alpha.
\]
The benchmark imposes $|c_\alpha|\leq10$ and
$V(x)-10^{-4}\|x\|^2\in\Sigma[x]$.  Each formulation maximizes
$\eta\geq0$ subject to an SOS certificate of degree $8$ for
\[
 p(x,\lambda)=-\nabla V(x)^\intercal F(x,\lambda)-\eta\|x\|^2.
\]
Since $V$ is quartic and the
vector field is cubic, $p$ has $x$-degree at most six and is affine in
$\lambda$, hence has total degree at most seven. Let $\eta^{\rm full}$ and $\eta^{\rm res}$ be the optimal values returned by the full and restricted formulations, respectively.

% We also use the vertex formulation
% \begin{equation}\label{eq:control-vertex-relaxation}
%  -\nabla V(x)^\intercal F_j(x)-\rho\|x\|^2
%  =\sigma_{0,j}(x)+\sigma_{x,j}(x)(R^2-\|x\|^2),
%  \qquad j=0,\ldots,r,
% \end{equation}
% where $\deg\sigma_{0,j}\leq8$ and
% $\deg\sigma_{x,j}\leq6$.  All vertex conditions share the same $V$ and
% $\rho$.  Because the derivative condition is affine in $\lambda$, robust
% nonnegativity on the simplex is equivalent to nonnegativity at its vertices.
% Moreover, evaluating a full certificate at a simplex vertex produces a
% feasible vertex certificate at the same SOS order.  Therefore, in exact
% arithmetic,
% \[
%  f^{\rm res}_4\leq f^{\rm full}_4\leq f^{\rm vert}_4.
% \]

Table~\ref{tab:control-performance}
reports median recorded times over instances
and their ratios for each configuration.
The two formulations produce essentially the same margins at the scale of
the experiment.
The optimized quartic Lyapunov functions substantially improve the known
quadratic reference margins from the range $1.2949$--$1.642$ to $2.0173$--$2.4385$ across the $20$ instances.
The model-size reduction is substantial. For $(n,r)=(4,3)$, the full formulation has a largest PSD block of size $330$,
compared with a largest block of size $70$ in the
restricted formulation.  For $(n,r)=(6,3)$, the largest block sizes are $715$ and $210$, respectively.
The smaller restricted SDPs also translate into lower recorded times
in every tested configuration.  Relative to the restricted formulation, the median solution time of the full formulation is between $15.4$ and $61.7$ times larger.

\begin{table}[t]
\centering
\caption{Comparison for robust Lyapunov design, $\theta\coloneqq \eta^{\rm full}-\eta^{\rm res}$.}
\label{tab:control-performance}
\setlength{\tabcolsep}{3pt}
\begin{tabular}{ccrrrrrr}
\toprule
$n$ & $r$ & $N_{\rm full}$ & $N_{\rm res}$
& $t_{\rm full}$ & $t_{\rm res}$ & speedup
& $\max|\theta|$ \\
\midrule
4 & 2 & 210 &  70 &   9.975 &  0.535 & $18.7\times$ & $2.41\times10^{-5}$ \\
4 & 3 & 330 &  70 &  39.936 &  0.647 & $61.7\times$ & $3.58\times10^{-5}$ \\
6 & 2 & 495 & 210 & 170.077 & 11.035 & $15.4\times$ & $6.35\times10^{-5}$ \\
6 & 3 & 715 & 210 & 895.424 & 39.484 & $22.7\times$ & $9.27\times10^{-5}$ \\
\bottomrule
\end{tabular}
\end{table}

\subsection{Min-max polynomial optimization}\label{subsec:minmax-experiments}
We finally consider the robust design problem
\begin{equation}\label{eq:minmax-benchmark}
\begin{aligned}
 \min_{z\in[-1,1]^4}\;
 \max_{\substack{x\in[-1,1]^n\\\lambda\in[-1,1]^r}}
 f(z,x,\lambda),\\
 f(z,x,\lambda)&=q_0(z,x)+\sum_{k=1}^r\lambda_kq_k(z,x),\\
 q_k(z,x)&=p_{0k}(x)+\sum_{j=1}^4 z_jp_{jk}(x),
 \qquad k=0,\ldots,r.
\end{aligned}
\end{equation}
Here the $p_{jk}$'s are dense cubic polynomials generated as in the box
benchmark.  For each
$(n,r)\in\{4,6,8\}\times\{2,4,6\}$, we generated five instances, giving $45$ instances in total.
The inner feasible set is independent of the design variable $z$.

We solve the full and restricted SOS relaxations for \eqref{eq:minmax-benchmark} at order $d=3$.
Table~\ref{tab:minmax-performance} reports median times over instances
for which both formulations succeeded, together with
the ratio of these medians. The full formulation ran out of memory on four instances with $(n, r) = (8, 6)$. The two bounds agree to within $10^{-5}$ on two instances, while the largest absolute difference is $1.3452$.  The relative difference
$(f_3^{\rm res}-f_3^{\rm full})/|f_3^{\rm full}|\times100\%$ has mean $1.209\%$,
median $0.700\%$, and maximum $6.537\%$; it is below $1\%$ on $26$ of
$41$ instances.
For $n=6$, increasing $r$ from $2$ to $6$
raises the median full time from $7.034$ to $521.029$ seconds, whereas
the restricted time changes from $0.795$ to $0.876$ seconds. Across the $41$ jointly solved instances, the restricted formulation reduces the total recorded
time from $9340.235$ to $112.400$ seconds, a factor of $83.1$.

\begin{table}[t]
\centering
\caption{Comparison for min-max polynomial optimization, $\theta\coloneqq f_3^{\rm res}-f_3^{\rm full}$.}
\label{tab:minmax-performance}
\setlength{\tabcolsep}{4pt}
\begin{tabular}{ccrrrrrr}
\toprule
$n$ & $r$ & $N_{\rm full}$ & $N_{\rm res}$
& $t_{\rm full}$ & $t_{\rm res}$ & speedup & $\max|\theta|$ \\
\midrule
4 & 2 &  84 &  35 &    0.714 & 0.074 &   $9.7\times$ & $5.37\times10^{-1}$ \\
4 & 4 & 165 &  35 &   13.023 & 0.102 & $128.1\times$ & $5.48\times10^{-1}$ \\
4 & 6 & 286 &  35 &  138.928 & 0.152 & $916.6\times$ & $1.35\times10^{0}$ \\
6 & 2 & 165 &  84 &    7.034 & 0.795 &   $8.9\times$ & $3.28\times10^{-1}$ \\
6 & 4 & 286 &  84 &   65.381 & 0.843 &  $77.6\times$ & $2.80\times10^{-1}$ \\
6 & 6 & 455 &  84 &  521.029 & 0.876 & $594.9\times$ & $8.07\times10^{-1}$ \\
8 & 2 & 286 & 165 &   63.073 & 8.747 &   $7.2\times$ & $1.93\times10^{-2}$ \\
8 & 4 & 455 & 165 &  481.614 & 8.940 &  $53.9\times$ & $3.89\times10^{-1}$ \\
8 & 6 & 680 & 165 & 2845.474 & 9.541 & $298.3\times$ & $5.09\times10^{-1}$ \\
\bottomrule
\end{tabular}
\end{table}

\vspace{1em}
\paragraph{Summary of the numerical results}
Across all benchmark families, the restricted formulation substantially
reduces the largest PSD block size, and the computational benefit generally becomes larger as the number of uncertainty parameters or the relaxation order increases.  The
median speedups range from $3.6\times$ to $532.5\times$ in the three direct
polynomial-optimization experiments, reach $448.2\times$ in the
relaxation-order comparison, range from $15.4\times$ to $61.7\times$ in the
robust Lyapunov experiment, and reach $916.6\times$ in the min--max
experiment.  The restricted formulation also solves all tested instances in
the direct optimization and min--max experiments, including several of the largest instances for which the full formulation ran out of memory.

These savings can come with weaker bounds at a fixed order.  For the box,
simplex, and $x$-dependent uncertainty experiments, the two formulations
agree numerically on most instances, while visible discrepancies occur on a
minority of cases; the relaxation-order experiment shows that such gaps can
decrease to near numerical precision after increasing the order.
The robust Lyapunov margins agree within $9.27\times10^{-5}$, whereas the
min--max experiment exhibits a more noticeable loss, with a mean relative
difference of $1.209\%$ and a maximum of $6.537\%$.  Overall, the experiments
show that the restricted formulation provides a substantial scalability advantage,
with a problem-dependent tradeoff between computational cost and relaxation strength.

%% file: sections/10_conclusions.tex
\section{Conclusions and further directions}\label{sec:conclusion}
We have developed a positivity theory that preserves affine
dependence on auxiliary parameters. Under the restricted Archimedean hypothesis, every strictly positive scalar polynomial admits an SOS certificate whose multipliers depend only on the principal variables.
This yields a convergent restricted moment-SOS hierarchy for polynomial optimization problems with affine parameters.
We have also extended the scalar theorem to polynomial-matrix constraints with any number of affine parameters and to matrix-valued objectives with one affine parameter. The numerical experiments have revealed substantial reductions in PSD block sizes and solution times, while the bound loss is often marginal.

Several directions remain open. 1) A quantitative degree bound for the restricted Positivstellens\"atze has not yet been derived. 2) Correlative and term sparsity could be combined to further reduce SDP sizes \cite{WakiEtAl2006,WangMagronLasserre2021}. 3) Can the restricted Positivstellens\"atze be generalized to other structured settings or to the noncommutaive case? 4) For matrix-valued objectives with several parameters, it would also be useful to identify intermediate structural assumptions---weaker than
uniform free positivity but stronger than scalar pointwise positivity---that
restore a restricted certificate theorem.